\documentclass[mathpazo]{cicp}
\usepackage{amsmath}
\usepackage{amssymb,amsfonts,amsthm}
\usepackage{graphicx}
\newtheorem{exam}{Example}[section]

\newcommand{\bF}{{\bf F}}
\newcommand{\bu}{{\bf u}}

\newcommand{\vq}{{\bf q}}

\newcommand{\bx}{{\bf x}}

\newcommand{\by}{{\bf y}}

\newcommand{\bv}{{\bf v}}

\newcommand{\bxi}{\boldsymbol{\xi}}
\newcommand{\be}{{\bf e}}

\def\T{{\mathcal T}}
\def\E{{\mathcal E}}

\def\bg{{\bf g}}
\def\bn{{\bf n}}

\def\3bar{{|\hspace{-.02in}|\hspace{-.02in}|}}

\newcommand{\cA}{\mathcal{A}}

\usepackage{color}

\begin{document}
\title{Discrete maximum principle for the weak Galerkin method for anisotropic diffusion problems}

\author[W. Huang et.~al.]{Weizhang Huang\affil{1} and
  Yanqiu Wang\affil{2}}
\address{\affilnum{1}\ Department of Mathematics, 
the University of Kansas, Lawrence, KS 66045, U.S.A. \\
\affilnum{2}\ Department of Mathematics, 
Oklahoma State University, Stillwater, OK 74078, U.S.A.}
\emails{{\tt whuang@ku.edu} (W.~Huang), {\tt yqwang@math.okstate.edu} (Y.~Wang)}
 
\begin{abstract}
A weak Galerkin discretization of the boundary value problem of a general anisotropic diffusion problem
is studied for preservation of the maximum principle. 
It is shown that the direct application of the $M$-matrix theory to the stiffness matrix of
the weak Galerkin discretization leads to a strong mesh condition requiring all of the mesh dihedral angles
to be strictly acute (a constant-order away from 90 degrees).
To avoid this difficulty, a reduced system is considered and shown to satisfy the discrete maximum
principle under weaker mesh conditions.
The discrete maximum principle is then established for the full weak Galerkin approximation
using the relations between the degrees of freedom located on elements and edges.
Sufficient mesh conditions for both piecewise constant and general anisotropic diffusion matrices are obtained.
These conditions provide a guideline for practical mesh generation for preservation of the maximum principle.
Numerical examples are presented.
\end{abstract} 


\maketitle

\section{Introduction}
We are concerned with the discrete maximum principle for a weak Galerkin discretization of the boundary
value problem (BVP) of a two-dimensional diffusion problem,
\begin{equation}
\label{eq:pde}
\begin{cases}
-\nabla\cdot(\cA\nabla u) = f, \qquad & \textrm{in }\Omega \\
u = g,\qquad &\textrm{on }\partial\Omega
\end{cases}
\end{equation}
where $\Omega\subset \mathbb{R}^2$ is a polygonal domain,
$f$ and $g$ are given functions, and $\cA$ is a symmetric and uniformly positive definite diffusion matrix
defined on $\Omega$.
The problem is isotropic when the diffusion matrix takes the form
$\cA = \alpha(\bx) I$ for some scalar function $\alpha(\bx)$
and anisotropic otherwise. In this work we are interested in the anisotropic situation.
It is known (e.g., see Evans \cite{Eva98}) that the classical solution of
(\ref{eq:pde}) satisfies the maximum principle,
\begin{equation}
f \le 0 \text{ in } \Omega \quad \Longrightarrow \quad \max_{\bx\in \Omega \cup \partial\Omega} u(\bx) = \max_{\bx\in\partial\Omega} u(\bx).
\label{eq:dmp}
\end{equation}

It is theoretically and practically important to investigate if a numerical approximation to (\ref{eq:pde})
preserves such a property. Indeed,
preservation of the maximum principle has attracted considerable attention from researchers; e.g., see
\cite{BKK08, Cia70, CR73, DDS04, Hua11, KL95, KK05, KKK07, KSS09, LH10, LiHu2012, LSS07, LS08,
LuHuQi2012, MD06, ShYu2011, Sto82, Sto86, SF73, Var66, WaZh11, XZ99, YuSh2008, ZZS2013}.
For example, it is shown by Ciarlet and Raviart \cite{CR73} and Brandts et al. \cite{BKK08}
that P1 conforming finite element (FE) solutions to isotropic diffusion problems
satisfy a discrete maximum principle (DMP)
if all of the mesh elements have nonobtuse dihedral angles. This nonobtuse angle condition
can be replaced in two dimensions
by a weaker condition (the Delaunay condition) \cite{SF73} requiring the sum of any pair of angles facing
a common interior edge to be less than or equal to $\pi$. For anisotropic diffusion problems,
Dr\v{a}g\v{a}nescu et al. \cite{DDS04} show that the nonobtuse angle condition fails to guarantee
the satisfaction of DMP for a  P1 conforming FE approximation.
Various techniques, including local matrix modification \cite{KSS09,LS08}, mesh optimization \cite{MD06},
and mesh adaptation \cite{LSS07}, have been proposed to reduce spurious oscillations.
More recently, it is shown by Li and Huang \cite{LH10} that P1 conforming FE solutions to anisotropic
diffusion problems can be guaranteed to satisfy DMP if the mesh satisfies an anisotropic nonobtuse angle condition
where mesh dihedral angles are measured in the metric specified by $\cA^{-1}$ instead of the Euclidean metric.
The result is extended to two dimensional problems \cite{Hua11}, problems with convection and reaction terms
\cite{LuHuQi2012}, and time dependent problems \cite{LiHu2012}. 
It is emphasized that while DMP has been well studied for conforming FE discretizations, 
it is less explored for nonconforming or mixed/mixed-hybrid FE methods.
Noticeably, DMP has been proven by Gu \cite{Gu} for a nonconforming FE discretization 
and by Hoteit et al. \cite{Hoteit} and Vohral\'{i}k and Wohlmuth \cite{Vohralik} for mixed/mixed-hybrid
FE discetizations. However, their results focus on isotropic diffusion problems.
Little is known about those discretizations for anisotropic diffusion problems.

The objective of this paper is to investigate the preservation of the maximum principle
by a weak Galerkin approximation of BVP (\ref{eq:pde}) with a general anisotropic diffusion matrix $\cA$.
The weak Galerkin method, recently introduced by Wang and Ye \cite{WangYe_PrepSINUM_2011},
is a FE method which uses a discontinuous FE space and approximates
derivatives with weakly defined ones on functions with discontinuity.
It can be easy to implement for meshes containing
arbitrary polygonal/polyhedral elements
\cite{WG-biharmonic, mwy-wg-stabilization, wy-mixed, WangYe_PrepSINUM_2011}.
The method has been successfully applied
to various model problems \cite{MuWangWangYe, mwy-wg-stabilization},
and its optimal order convergence has been established for second order elliptic equations
\cite{mwy-wg-stabilization, wy-mixed, WangYe_PrepSINUM_2011}.
On the other hand, the weak Galerkin method has not been studied in the aspect of preserving
the maximum principle. Such studies are useful in practice to avoid unphysical numerical solutions.
They are also beneficial in theory since they provide in-depth understandings of the newly developed
weak Galerkin method. It should be pointed out that such studies are not trivial.
A commonly used and effective tool in the study of preservation of the maximum
principle is the theory of $M$-matrices, matrices in the form of $s I - B$, where
$I$ is the identity matrix of some order $n > 0$, $s$ is a positive number, and $B$ is a nonnegative
matrix ($B(i,j)\ge 0$) with spectral radius less than $s$.
In principle, the theory can be directly applied to the current situation where
the weak Galerkin method defines the degrees of freedom separately on edges and inside elements.
(For the current work, we consider a simplest and lowest order weak Galerkin
method where solutions are approximated using functions that are piecewise constant on edges and inside elements.)
Unfortunately, this direct application leads to a strong mesh condition requiring all of
the mesh dihedral angles to be strictly acute ($\mathcal{O}(1)$ smaller than 90 degrees)
for DMP satisfaction (cf. Remark~\ref{rem:Mbb}).
To avoid this difficulty, 
we use a two-step procedure to study DMP preservation.
We first obtain a reduced system involving only the degrees of freedom on edges, 
and show that it satisfies DMP
if the mesh is sufficiently fine and meets an $\mathcal{O}(h^2)$-acute anisotropic angle condition
(which requires the angles to be only $\mathcal{O}(h^2)$ away from 90 degrees), where $h$ is the maximal
element diameter.
We then show that the weak Galerkin approximation to the solution on elements also satisfies DMP.
The mesh condition provides a guideline for practical generation of DMP-preserving meshes
for the weak Galerkin discretization of general anisotropic diffusion problems.

An outline of the paper is given as follows. A weak Galerkin discretization for BVP (\ref{eq:pde})
is given in \S\ref{sec:wg}. A weak gradient is defined and the properties of the discrete system are discussed in
\S\ref{sec:properties}. Preservation of the maximum principle is studied in \S\ref{sec:DMP},
followed by numerical examples in \S\ref{sec:numerical}. Finally, conclusions are drawn in \S\ref{sec:conclusion}.

\section{The weak Galerkin formulation}
\label{sec:wg}

In this section we describe a simplest and lowest order weak Galerkin discretization for BVP (\ref{eq:pde}).

We start with introducing some notation.
For any given polygonal domain $D$, we use the standard notation for 
Sobolev spaces $H^s(D)$ and $H_0^s(D)$ with $s\ge 0$.
The inner-product, norm, and semi-norms in $H^s(D)$ are denoted by
$(\cdot,\cdot)_{s,D}$, $\|\cdot\|_{s,D}$, and $|\cdot|_{r,D}$ ($0\le r\le s$), respectively.
When $s=0$, $H^0(D)$ coincides with $L^2(D)$, the space
of square integrable functions. In this case, the subscript
$s$ is suppressed from these notation. So is the subscript $D$ when
$D=\Omega$. For $s<0$, the space $H^s(D)$ is defined as the dual
of $H_0^{-s}(D)$. The above notation is extended in a straightforward manner
to vector-valued and matrix-valued functions and to an edge, a domain with a lower
dimension. Particularly, $\|\cdot\|_{s,e}$ and $\|\cdot\|_e$ denote the norm in
$H^s(e)$ and $L^2(e)$, respectively.
Functions or quantities  on the boundary of $\Omega$ or boundary edges of a mesh
will be denoted by $(\cdot)^{\partial}$.

The variational form of BVP (\ref{eq:pde}) reads as:
Given $f\in H^{-1}(\Omega)$ and $g\in H^{\frac{1}{2}}(\partial\Omega)$, find $u\in H^1(\Omega)$ such that
$u=g$ on $\partial\Omega$ and
\begin{equation}
\label{eq:weakformulation}
(\cA \nabla u,\, \nabla v) = \langle f,\, v \rangle, \qquad \forall v\in H_0^1(\Omega)
\end{equation}
where $\langle\cdot,\cdot\rangle$ denotes the duality form on $\Omega$.

To define the weak Galerkin approximation of (\ref{eq:weakformulation}), let ${\cal T}_h$ be a given triangular mesh
on $\Omega$.
For each triangle $K\in {\cal T}_h$, denote the interior and
boundary of $K$ by $K_0$ and $\partial K$, respectively.
Also, denote the diameter (i.e., the length of the longest edge) of $K$ by $h_K$
and let $h=\max_{K\in\mathcal{T}_h}h_K$. The boundary $\partial K$ of $K$ consists of three edges.
Denote by $\E_h$ the set of all edges in ${\cal T}_h$. For simplicity, hereafter we use ``$\lesssim$''
to denote ``less than or equal to up to a general constant
independent of the mesh size or functions appearing in the inequality''.
We denote
by $P_0(K_0)$ the set of constant polynomials on the interior $K_0$ of triangle $K$.
Likewise, $P_{0}(e)$ is the set of constant polynomials on $e \in \E_h$.
Following \cite{WangYe_PrepSINUM_2011}, we define a weak discrete space on mesh $\T_h$ by
$$
V_{h} =  \{v:\:  v|_{K_0}\in P_0(K_0)\textrm{ for } K\in \T_h;
\ v|_e\in P_0(e)\textrm{ for } e\in \E_h\}.
$$
Note that $V_h$ does not require the continuity
of its functions across interior edges. A function in $V_h$ is
characterized by its values ($v_0$) on the interior of the elements and those ($v_b$)
on edges. It is often convenient to represent it
with two components, $v =\{v_0, v_b\}$.  $V_h$ is one of the lowest order weak Galerkin
space defined on triangular meshes \cite{WangYe_PrepSINUM_2011}.
To cope with the boundary conditions, for a given piecewise constant function $g_h$ defined on $\E_h\cap\partial\Omega$
we denote
$$
V_{h}^{g_h} = \{v:\: v\in V_h \textrm{ and } v_b|_e = g_h|_e \textrm{ for } e\in \E_h\cap\partial\Omega\} .
$$
When $g_h\equiv 0$, $V_{h}^{g_h}$ becomes $V_{h}^{0}$.

The weak Galerkin method seeks an approximation $u_h\in V_h^{g_h}$ to the solution of (\ref{eq:weakformulation}),
where $g_h$ is an approximation to the actual boundary data $g$. Notice that $V_h \not\subset H^1(\Omega)$ and
the gradient operator is not defined for functions in $V_h$.
For the moment we assume that a weak gradient, denoted by $\nabla_w$, is defined for functions in $V_h$.
(A definition will be given in the next section.)
Then, a weak Galerkin FE approximation is defined as $u_h = \{u_0,\, u_b\}\in V_h^{g_h}$ such that
\begin{equation}
\label{eq:wg}
  (\cA \nabla_w u_h,\, \nabla_w v_h) = \langle f,\, v_0\rangle,  \qquad \forall \; v_h = \{v_0,\, v_b\}\in V_h^0.
\end{equation}
The well-posedness and error estimates of the weak Galerkin formulation (\ref{eq:wg}) have been discussed
in \cite{WG-biharmonic,WangYe_PrepSINUM_2011}.

Equation (\ref{eq:wg}) can be cast in a matrix form.
Denote the numbers of the triangles, interior edges, and boundary edges in $\T_h$
by $N_0$,  $N_b$, and $N_b^{\partial}$, respectively.
Let $\phi_{0,i}$  ($i=1,\ldots,N_0$) be the basis function in $V_h$ associated with the $i^{\text{th}}$ element
such that  its value is $1$ on the triangle and $0$ on other elements or all edges. 
Similarly, let $\phi_{b,i}$ ($i=1,\ldots,N_b$) and $\phi_{b,i}^{\partial}$ ($i=1,\ldots,N_b^{\partial}$)
be the basis functions in $V_h$ associated with $i^{\text{th}}$ interior and boundary edges, respectively.
Then $u_h=\{u_0,u_b\}\in V_h$ can be expressed as
\begin{equation}
u_h = \sum_{i=1}^{N_0} u_{0,i} \phi_{0,i} + \sum_{i=1}^{N_b} u_{b,i} \phi_{b,i}
+ \sum_{i=1}^{N_b^{\partial}} u_{b,i}^{\partial} \phi_{b,i}^{\partial}.
\label{eq:uh}
\end{equation}
For convenience, we define the vector representation of $u_h$ as
$$
\bu = \begin{bmatrix} \bu_0 \\[1mm] \bu_b
\\[1mm] \bu_b^{\partial} \end{bmatrix},
\quad \text{ with }\quad 
\bu_0 = \begin{bmatrix}u_{0,1} \\[1mm] u_{0,2} \\ \vdots\\[1mm] u_{0,N_0}\end{bmatrix},\qquad
\bu_b = \begin{bmatrix} u_{b,1} \\[1mm] u_{b,2} \\ \vdots \\[1mm] u_{b,N_b} \end{bmatrix},\qquad
\bu_b^{\partial} = \begin{bmatrix} u_{b,1}^{\partial} \\[1mm] u_{b,2}^{\partial} \\ \vdots \\[1mm]
    u_{b,N_b^{\partial}}^{\partial} \end{bmatrix} .
$$
Inserting (\ref{eq:uh}) into (\ref{eq:wg}) and taking $v_h$ to be the basis functions, we get
\begin{equation}
\label{eq:matrixform}
M \bu = \bF,
\end{equation}
where
\[
M = \begin{bmatrix} M_{0,0} & M_{0,b} & M_{0,b}^{\partial} \\[1mm] M_{b,0} & M_{b,b} & M_{b,b}^{\partial}
        \\[1mm] 0 & 0 & I\end{bmatrix} ,
\qquad \bF = \begin{bmatrix}\bF_0 \\[1mm] 0 \\[1mm] \bg_h \end{bmatrix},
\]
$$
\begin{aligned}
M_{0,0} &= \begin{bmatrix}(\cA\nabla_w \phi_{0,j},\, \nabla_w \phi_{0,i})\end{bmatrix} \in \mathbb{R}^{N_0\times N_0}, \qquad &
M_{b,0} &= \begin{bmatrix}(\cA\nabla_w \phi_{0,j},\, \nabla_w \phi_{b,i})\end{bmatrix} \in \mathbb{R}^{N_b\times N_0}, \\
M_{0,b} &= \begin{bmatrix}(\cA\nabla_w \phi_{b,j},\, \nabla_w \phi_{0,i})\end{bmatrix} \in \mathbb{R}^{N_0\times N_b}, \qquad &
M_{b,b} &= \begin{bmatrix}(\cA\nabla_w \phi_{b,j},\, \nabla_w \phi_{b,i})\end{bmatrix} \in \mathbb{R}^{N_b\times N_b}, \\
M_{0,b}^{\partial} &= \begin{bmatrix}(\cA\nabla_w \phi_{b,j}^{\partial},\, \nabla_w \phi_{0,i})\end{bmatrix}
  \in \mathbb{R}^{N_0\times N_b^{\partial}}, \qquad &
M_{b,b}^{\partial} &= \begin{bmatrix}(\cA\nabla_w \phi_{b,j}^{\partial},\, \nabla_w \phi_{b,i})\end{bmatrix}
   \in \mathbb{R}^{N_b\times N_b^{\partial}}, \\
\bF_0 &= \begin{bmatrix}\langle f,\, \phi_{0,i}\rangle\end{bmatrix}  \in \mathbb{R}^{N_0},
\end{aligned}
$$
and $\bg_h\in \mathbb{R}^{N_b^{\partial}}$ is the vector representation of the discrete boundary data $g_h$.

We are interested in the preservation of the maximum principle by the weak Galerkin approximation defined above.
A commonly used and effective tool for this type of study is the theory of $M$-matrices. 
In principle, the theory can be directly applied to the system (\ref{eq:matrixform}). However,
as will be seen in Remark~\ref{rem:Mbb}, the mesh condition ensuring all of the off-diagonal entries of
$M_{b,b}$ to be nonpositive is generally stronger than that obtained with a reduced system.
Such reduced system is obtained by eliminating the variable $\bu_0$ in (\ref{eq:matrixform}), i.e., 
\begin{equation}
\label{eq:SchurComplement-1}
\begin{bmatrix}M_{b,b} - M_{b,0}M_{0,0}^{-1}M_{0,b} & M_{b,b}^{\partial} - M_{b,0}M_{0,0}^{-1} M_{0,b}^{\partial} \\[1mm]
0 & I \end{bmatrix}
\begin{bmatrix}\bu_b  \\[1mm] \bu_b^{\partial} \end{bmatrix} =
\begin{bmatrix}-M_{b,0}M_{0,0}^{-1} \bF_0 \\[1mm] \bg_h\end{bmatrix} .
\end{equation}
In the next section, we shall show that the stiffness matrix of (\ref{eq:SchurComplement-1})
can be an $M$-matrix under suitable, weaker mesh conditions.
Notice that the stiffness matrix involves the inverse of $M_{0,0}$. It is easy to see that $M_{0,0}$ is diagonal
since the support of any basis function $\phi_{0, i}$ does not overlap with the support
of other basis functions $\phi_{0,j}$ with $j \neq i$. Thus, the involvement of the inverse of $M_{0,0}$
will not complicate the analysis of the system. More properties of (\ref{eq:SchurComplement-1})
are discussed in the next section.

For convenience, we rewrite (\ref{eq:SchurComplement-1}) as
\begin{equation}
\label{eq:SchurComplement}
\begin{bmatrix}A & A^{\partial}\\ 0& I\end{bmatrix}
\begin{bmatrix}\bu_b  \\[1mm] \bu_b^{\partial} \end{bmatrix} =
\begin{bmatrix}-M_{b,0}M_{0,0}^{-1} \bF_0 \\[1mm] \bg_h\end{bmatrix} ,
\end{equation}
where
\[
\bar{A} = \begin{bmatrix}A & A^{\partial}\\ 0& I\end{bmatrix},\quad
A = M_{b,b} - M_{b,0}M_{0,0}^{-1}M_{0,b},\quad
A^{\partial} = M_{b,b}^{\partial} - M_{b,0}M_{0,0}^{-1} M_{0,b}^{\partial} .
\]

\section{Weak gradient and properties of the discrete system}
\label{sec:properties}

In this section we present a definition of the weak gradient operator and study the properties of
the discrete system (\ref{eq:SchurComplement}).

We use a definition of the weak gradient operator proposed in \cite{WangYe_PrepSINUM_2011}.
For any element $K\in \T_h$, we denote the space of the lowest order Raviart-Thomas element \cite{RT77} on $K$
by $RT_0(K)$, i.e., 
$$
RT_0(K) = (P_0(K))^2 + \bx P_0(K).
$$
The degrees of freedom of $RT_0(K)$ consist of $0^{\text{th}}$ order moments of normal
components on each edge of $K$. The functions in $RT_0(K)$ can be written in the
form of $c(\bx - \bx_0)$ for some constant $c$ and vector $\bx_0$. 
Define
$$
\Sigma_h = \{\vq\in (L^2(\Omega))^2:\: \vq|_K \in RT_0(K)\textrm{ for } K\in\T_h\}.
$$
A discrete weak gradient
\cite{WangYe_PrepSINUM_2011} of $v_h=\{v_0,v_b\}\in V_h$ is defined
as $\nabla_w v_h \in \Sigma_h$ such that on each $K\in\T_h$,
\begin{equation}
\label{discrete-weak-gradient-new}
(\nabla_w v_h,\, \vq)_K = -(v_0,\, \nabla\cdot \vq)_K + \langle v_b,\, \vq\cdot\bn\rangle_{\partial K},
\quad \forall \; \vq\in RT_{0}(K)
\end{equation}
where $\bn$ is the unit outward normal on $\partial K$.
Such a discrete weak gradient is well defined on $V_h$. Moreover,
$\nabla_w \phi_{0,i}$, $\nabla_w \phi_{b,i}$, and $\nabla_w \phi_{b,i}^{\partial}$
can be found explicitly. To this end, we denote the centroid and area of $K\in \T_h$ by
$\bx_K$ and $|K|$, respectively, and the length of $e\in \E_h$ by $|e|$.

\begin{lemma}
\label{lem:basis0}
Letting $K$ be the $i^{\text{th}}$ triangle in $\T_h$, then
\[
\nabla_w\phi_{0,i}|_K = -C_K (\bx - \bx_K),
\]
where
\begin{equation}
C_K = \frac{2|K|}{\|\bx-\bx_K\|_K^2}.
\label{eq:ck}
\end{equation}
\end{lemma}

\begin{proof}
Taking $v_h = \phi_{0,i}$ and $\vq = \begin{bmatrix}1\\0\end{bmatrix}$ and 
$\begin{bmatrix}0\\1\end{bmatrix}$ in (\ref{discrete-weak-gradient-new}), we have
$$
\begin{aligned}
\left(\nabla_w \phi_{0,i},\, \begin{bmatrix}1\\0\end{bmatrix}\right)_K &=
   -\left(\phi_{0,i},\, \nabla\cdot \begin{bmatrix}1\\0\end{bmatrix}\right)_K = 0, \\[2mm]
\left(\nabla_w \phi_{0,i},\, \begin{bmatrix}0\\1\end{bmatrix}\right)_K &=
   -\left(\phi_{0,i},\, \nabla\cdot \begin{bmatrix}0\\1\end{bmatrix}\right)_K = 0,
\end{aligned}
$$
which implies $\int_K \nabla_w \phi_{0,i}\, d\bx = \mathbf{0}$.
Since both components of $\nabla_w \phi_{0,i}$ are linear polynomials, we get
$\nabla_w \phi_{0,i} = c(\bx - \bx_K)$ for some constant $c$. To determine $c$, 
we take $\vq = \bx - \bx_K$ in (\ref{discrete-weak-gradient-new}) and have
$$
  \left(\nabla_w \phi_{0,i},\, \bx - \bx_K\right)_K = -\left(\phi_{0,i},\, \nabla\cdot (\bx - \bx_K)\right)_K
      = -2|K|.
$$
Combining this with $\nabla_w \phi_{0,i} = c(\bx - \bx_K)$, we obtain $c = C_K$.
\end{proof}

\begin{remark}
From the definition of $\bx_K$, one can easily see that
\begin{equation}
\label{eq:basis0}
\left(\bx - \bx_K,\, \begin{bmatrix}a\\b\end{bmatrix}\right)_K = 0,
\quad \forall \; \begin{bmatrix}a\\b\end{bmatrix} \in \mathbb{R}^2 .
\end{equation}
This can also be verified by direct calculation.
Identity (\ref{eq:basis0}) will be used frequently in the following analysis.
\qed
\end{remark}

\begin{lemma}
\label{lem:basisb}
Assume that the $i^{\text{th}}$ interior edge $e_i$ is on $\partial K$. Then,
\[
\nabla_w \phi_{b,i}|_K = \frac{C_K}{3} (\bx - \bx_K) + \frac{|e_i|}{|K|}\bn_{i,K},
\]
where $\bn_{i,K}$ is the unit outward normal on $e_i$ with respect to $K$ and $C_K$ is given in (\ref{eq:ck}).
The formula also applies to the boundary edge $e_i^{\partial}$, viz.,
\[
\nabla_w \phi_{b,i}^{\partial}|_K = \frac{C_K}{3} (\bx - \bx_K) + \frac{|e_i^{\partial}|}{|K|}\bn_{i,K}^{\partial} .
\]
\end{lemma}

\begin{proof}
We only consider $\nabla_w \phi_{b,i}|_K$ since the proof for $\nabla_w \phi_{b,i}^{\partial}|_K$ is exactly the same.
Taking $v_h = \phi_{b,i}$ and $\vq = \begin{bmatrix}1\\0\end{bmatrix}$ and 
$\begin{bmatrix}0\\1\end{bmatrix}$ in (\ref{discrete-weak-gradient-new}), we have
$$
\begin{aligned}
\left(\nabla_w \phi_{b,i},\, \begin{bmatrix}1\\0\end{bmatrix}\right)_K &=
   \left\langle\phi_{b,i},\, \begin{bmatrix}1\\0\end{bmatrix}\cdot \bn_{i,K} \right\rangle_{e_i}
   = |e_i| \; \begin{bmatrix}1\\0\end{bmatrix}\cdot \bn_{i,K}, \\[2mm]
\left(\nabla_w \phi_{b,i},\, \begin{bmatrix}0\\1\end{bmatrix}\right)_K &=
   \left\langle\phi_{b,i},\, \begin{bmatrix}0\\1\end{bmatrix}\cdot \bn_{i,K} \right\rangle_{e_i}
   = |e_i| \; \begin{bmatrix}0\\1\end{bmatrix}\cdot \bn_{i,K},
\end{aligned}
$$
which implies $\int_K \nabla_w \phi_{b,i}\, d\bx = |e_i|\bn_{i,K}$.
Again, since both components of $\nabla_w \phi_{b,i}$ are linear polynomials, we get
$$
\nabla_w \phi_{b,i}|_K = c (\bx - \bx_K) + \frac{|e_i|}{|K|}\bn_{i,K}.
$$
To determine $c$, we take $\vq = \bx - \bx_K$ in (\ref{discrete-weak-gradient-new}) and get
$$
\left( c (\bx - \bx_K) + \frac{|e_i|}{|K|}\bn_{i,K},\, \bx - \bx_K  \right)_K
   = \left\langle \phi_{b,i},\, (\bx - \bx_K) \cdot\bn_{i,K} \right\rangle_{e_i}.
$$
From (\ref{eq:basis0}), the left-hand side of the above equation becomes $c\|\bx-\bx_K\|_K^2$.
For the right-hand side, we observe that
for any $\bx\in e_i$, $(\bx - \bx_K) \cdot\bn_{i,K}$ is equal to one third
of the height of triangle $K$ with $e_i$ as the base. This implies that the right-hand side
is equal to $\frac{2}{3}|K|$. Combining these results, we obtain $c= C_K/3$ and therefore
the expression for $\nabla_w \phi_{b,i}|_K$.
\end{proof}

\begin{remark}
\label{rem:basis0b}
From the definitions of $\phi_{0,i}$ and $\phi_{b,i}$ and Lemmas~\ref{lem:basis0} and \ref{lem:basisb},
we see that the support of $\nabla_w \phi_{0,i}$ consists of the $i^{\text{th}}$ element and
that of $\nabla_w \phi_{b,i}$ consists of the elements sharing $e_i$ as a common edge.
\qed
\end{remark}

Having obtained $\nabla_w \phi_{0,i}$, $\nabla_w \phi_{b,i}$, and $\nabla_w \phi_{b,i}^{\partial}$, we now are
ready to find the matrices $M_{0,0}$, $M_{0,b}$, $M_{0,b}^\partial$, $M_{b,0}$, $M_{b,b}$, and $M_{b,b}^\partial$.

\begin{lemma}
\label{lem:M00}
$M_{0,0}$ is a diagonal matrix with diagonal entries
$$
M_{0,0}(i,i) = C_K^2 \|\bx - \bx_K\|_{\cA,K}^2,
$$
where $K$ is the $i^{\text{th}}$ triangle and 
\begin{equation}
\|\bx - \bx_K\|_{\cA,K} = (\cA(\bx - \bx_K),\, \bx - \bx_K)_K^{\frac 1 2}.
\label{AX-1}
\end{equation}
\end{lemma}

\begin{proof}
$M_{0,0}$ is diagonal since the support of the basis function $\phi_{0, i}$ does not overlap with the support
of other basis functions $\phi_{0,j}$ with $j \neq i$. Moreover, from Lemma \ref{lem:basis0},
\begin{align*}
M_{0,0}(i,i) &= (\cA \nabla_w \phi_{0,i},\, \nabla_w \phi_{0,i})_K \\
 &= \left( -\cA C_K(\bx - \bx_K),\, -C_K(\bx - \bx_K) \right)_K \\
 &= C_K^2 \|\bx - \bx_K\|_{\cA,K}^2.
\end{align*}
\end{proof}

\begin{lemma}
\label{lem:Mb0}
If the $i^{\text{th}}$ interior edge $e_i$ is an edge of the $j^{\text{th}}$ triangle $K\in \T_h$, then
$$
M_{b,0}(i,j) = M_{0,b}(j,i) = -\frac{1}{3}C_K^2 \|\bx - \bx_K\|_{\cA,K}^2
- C_K \frac{|e_i|}{|K|} \left(\cA (\bx - \bx_K),\, \bn_{i,K} \right)_K.
$$
Otherwise, $M_{b,0}(i,j) = M_{0,b}(j,i) = 0$.
Similarly, if the $i^{\text{th}}$ boundary edge $e_i^{\partial}$  is an edge of the $j^{\text{th}}$ triangle $K\in \T_h$, then
$$
M_{0,b}^{\partial}(j,i) = -\frac{1}{3}C_K^2 \|\bx - \bx_K\|_{\cA,K}^2
- C_K \frac{|e_i^{\partial}|}{|K|} \left(\cA(\bx - \bx_K),\, \bn_{i,K}^{\partial} \right)_K.
$$
Otherwise, $M_{0,b}^{\partial}(j,i) = 0$.
\end{lemma}

\begin{proof}
From Remark~\ref{rem:basis0b}, we know $M_{b,0}(i,j) = M_{0,b}(j,i) = 0$ when
$e_i$ is not an edge of the $j$th triangle $K$. On the other hand,
when $e_i$ is an edge of the $j$th triangle $K$, by Lemmas \ref{lem:basis0} and \ref{lem:basisb} we have
$$
\begin{aligned}
M_{b,0}(i,j) =M_{0,b}(j,i) &= (\cA \nabla_w \phi_{0,j},\, \nabla_w \phi_{b,i})_K \\
  &= \left(-\cA C_K(\bx - \bx_K),\, \frac{1}{3}C_K(\bx - \bx_K)+ \frac{|e_i|}{|K|}\bn_{i,K} \right)_K \\
  &= -\frac{1}{3} C_K^2 \|\bx - \bx_K\|_{\cA,K}^2 - C_K \frac{|e_i|}{|K|} \left(\cA (\bx - \bx_K),\, \bn_{i,K} \right)_K.
\end{aligned}
$$
This completes the proof for $M_{b,0}(i,j)$ and $M_{0,b}(j,i)$.
The proof for $M_{0,b}^{\partial}(j,i)$ is similar.
\end{proof}

For the calculation of $M_{b,b}$ and $M_{b,b}^{\partial}$, we need to know how many elements 
are sharing a given edge.
We first consider the situation of two interior edges which can be the same.
Denote by $\T_{i,j}$ the collection of triangles in $\T_h$ that contain both $e_i$ and $e_j$ as edges, i.e.,
\[
\T_{i,j} = \{ K \in \T_h:\; e_i, e_j \in \partial K \} .
\]
When $e_i$ and $e_j$ are the same, $T_{i,j}$ contains two elements sharing the edge.
If they are not the same, they can be either the edges of a triangle or two different triangles.
$T_{i,j}$ contains an element in the former case and none in the latter. To summarize, $T_{i,j}$ is given by
\[
\T_{i,j} = \begin{cases}  \{ K, K'\}, &\text{ for } e_i = e_j \text{ (where $K$ and $K'$ are elements satisfying
					 $e_i = e_j\in \partial K \cap \partial K'$)} \\
\{ K \} , &\text{ for } e_i \ne e_j \text{ and if there exists an element $K$ such that } e_i, e_j \in \partial K \\
\emptyset , &\text{ for } e_i \ne e_j \text{ and if there is no element $K$ such that } e_i, e_j \in \partial K. \\
\end{cases}
\]
For the situation where $e_i$ is an interior edge and $e_j^\partial$ is a boundary edge, $\T_{i,j}^\partial$ contains
at most one triangle in $\T_h$ that takes both $e_i$ and $e_j^\partial$ as its edges.

\begin{lemma}
\label{lem:Mbb}
Letting $e_i$ and $e_j$ be two interior edges, then
\begin{align*}
M_{b,b}(i,j) & = \sum_{K\in \T_{i,j}} \left[ \frac{C_K^2}{9} \|\bx - \bx_K\|_{\cA,K}^2
+ \frac{C_K}{3|K|} (\cA (\bx - \bx_K),\, |e_i|\bn_{i,K} + |e_j|\bn_{j,K})_K \right.
\\
& \qquad \qquad \left. \frac{}{}
+\; \frac{|e_i|\, |e_j|}{|K|^2} (\cA\bn_{j,K},\, \bn_{i,K})_K \right] .
\end{align*}
For the case where $e_i$ is an interior edge and $e_j^{\partial}$ is a boundary edge, 
\begin{align*}
M_{b,b}^{\partial}(i,j) & = \sum_{K\in \T_{i,j}^{\partial}} \left[ \frac{C_K^2}{9} \|\bx - \bx_K\|_{\cA,K}^2
+ \frac{C_K}{3|K|} (\cA (\bx - \bx_K),\, |e_i|\bn_{i,K} + |e_j^{\partial}|\bn_{j,K}^{\partial})_K \right.
\\
& \qquad \qquad \left. \frac{}{}
+\; \frac{|e_i|\, |e_j^{\partial}|}{|K|^2} (\cA\bn_{j,K}^{\partial},\, \bn_{i,K})_K \right] .
\end{align*}
\end{lemma}

\begin{proof}
The results follow directly from Lemma \ref{lem:basisb} and
\[
M_{b,b}(i,j) = \sum_{K\in \T_{i,j}} (\cA \nabla_w \phi_{b,j},\, \nabla_w \phi_{b,i})_K,
\quad 
M_{b,b}^{\partial}(i,j) = \sum_{K\in \T_{i,j}^{\partial}} (\cA \nabla_w \phi_{b,j}^{\partial},\, \nabla_w \phi_{b,i})_K .
\]
\end{proof}

\begin{remark} \label{rem:Mbb}
From this lemma, we can see that the requirement of the off-diagonal entries of $M_{b,b}$ being nonpositive yields
the mesh condition
\[
\frac{C_K^2}{9} \|\bx - \bx_K\|_{\cA,K}^2
+ \frac{C_K}{3|K|} (\cA (\bx - \bx_K),\, |e_i|\bn_{i,K} + |e_j|\bn_{j,K})_K
+ \frac{|e_i|\, |e_j|}{|K|^2} (\cA\bn_{j,K},\, \bn_{i,K})_K \le 0 .
\]
To get a feel for this condition, we consider a simple situation with $\cA = I$ and $e_i$ and $e_j$ being
two different edges of a triangle $K$. From (\ref{eq:ck}) and (\ref{eq:basis0}), the inequality reduces to
\[
\frac{4 |K|^2}{9 \|\bx - \bx_K\|_{K}^2}
+ \frac{|e_i|\, |e_j|}{|K|} \bn_{j,K} \cdot \bn_{i,K} \le 0 .
\]
Denote the internal angle of $K$ formed by edges $e_i$ and $e_j$ by $\theta$.
From $|K| = \frac{1}{2} |e_i|\, |e_j| \sin\theta$, the above condition becomes
\begin{equation} \label{eq:fullsystemmeshcondition}
\cot\theta \ge \frac{2|K|^2}{9\|\bx-\bx_K\|_K^2}.
\end{equation}
Since the right-hand side has a lower bound
\[
\frac{2|K|^2}{9\|\bx-\bx_K\|_K^2} \ge \frac{2|K|^2}{9\int_K h_K^2\, dx} = \frac{2|K|}{9h_K^2},
\]
therefore, for an element with $|K| = \mathcal{O}(h_K^2)$ the mesh condition requires $\cot\theta$ to be $\mathcal{O}(1)$
away from zero, i.e., $\theta$ be $\mathcal{O}(1)$ away from $\pi/2$.
This is much stronger than that to be obtained with the reduced system.
As will be seen in Theorem~\ref{thm:pwconstant} in the next section, the mesh condition
obtained with the reduced system only requires $\theta$ to be nonobtuse for the current situation $\cA = I$
(and for a more general situation with piecewise constant $\cA$).
\end{remark}

\begin{lemma}
\label{lem:Aelem}
For any two interior edges $e_i$ and $e_j$,
$$
A(i,j) = \sum_{K\in \T_{i,j}}  \frac{|e_i|\, |e_j|}{|K|^2} \left[ (\cA\bn_{j,K},\, \bn_{i,K})_K -
\frac{(\cA(\bx-\bx_K),\, \bn_{i,K})_K\, (\cA(\bx-\bx_K),\, \bn_{j,K})_K}{\|\bx-\bx_K\|_{\cA,K}^2} \right] .
$$
Moreover, for any interior edge $e_i$ and boundary edge $e_j^{\partial}$,
$$
A^{\partial}(i,j) = \sum_{K\in \T_{i,j}^{\partial}}  \frac{|e_i|\, |e_j^{\partial}|}{|K|^2} \left[ (\cA\bn_{j,K}^{\partial},\, \bn_{i,K})_K -
\frac{(\cA(\bx-\bx_K),\, \bn_{i,K})_K\, (\cA(\bx-\bx_K),\, \bn_{j,K}^{\partial})_K}{\|\bx-\bx_K\|_{\cA,K}^2} \right] .
$$
\end{lemma}

\begin{proof}
Denote by $K_k$ the $k^{\text{th}}$ triangle in $\T_h$.
Using the sparsity pattern of $M_{b,0}$, $M_{0,0}$ and $M_{0,b}$, it is not hard to see that
$$
A(i,j) = M_{b,b}(i,j) - \sum_{K_k\in \T_{i,j}}  M_{b,0}(i,k)\, M_{0,0}^{-1}(k,k)\, M_{0,b}(k,j) .
$$
The rest is a straightforward calculation using Lemmas \ref{lem:Mb0} and \ref{lem:Mbb}.
The calculation of $A^{\partial}(i,j)$ is similar.
\end{proof}

\begin{remark}
\label{rem:pwconst-Asign}
If the diffusion coefficient $\cA$ is piecewise constant on $\T_h$, then by (\ref{eq:basis0}) we have
$$
A(i,j) = \sum_{K\in \T_{i,j}}  \frac{|e_i|\, |e_j|}{|K|^2} (\cA\bn_{j,K},\, \bn_{i,K})_K,
\quad 
A^{\partial}(i,j) = \sum_{K\in \T_{i,j}^{\partial}}  \frac{|e_i|\, |e_j^{\partial}|}{|K|^2} (\cA\bn_{j,K}^{\partial},\, \bn_{i,K})_K .
$$
In this case, it is not difficult to see that system (\ref{eq:SchurComplement}) is exactly the same as
the discrete system of the lowest order Crouzeix-Raviart element ($P_1$ nonconforming FE)
for (\ref{eq:weakformulation}). Since the mixed-hybrid FE discretization is also equivalent to
the $P_1$ nonconforming FE discretization when $\cA$ is piecewise constant \cite{Arnold85},
we know that the weak Galerkin method is equivalent to the mixed-hybrid FE discretization in this case.
This implies that for piecewise constant $\cA$, our DMP analysis also applies to 
$P_1$ nonconforming and mixed-hybrid FE discretizations which have been studied very little in the past
for anisotropic diffusion problems.

It should also be pointed out that the equivalence is valid only in the sense that
the weak Galerkin solution $u_b$ on edges is equal
to the Lagrange multiplier used in the mixed-hybrid FE discretization.
On the other hand, the flux $\cA\nabla_w u_h$ and
the values of $u_0$ in the weak Galerkin method are generally different from
the dual and primal variables in the mixed-hybrid FE discretization. 
They are identical only when $\cA$ is of the form $c I$ for some constant $c$.
\qed
\end{remark}

\begin{remark}
When $\cA$ is not piecewise constant, the weak Galerkin method is generally different from the 
nonconforming or mixed-hybrid FE method.  The difference between the weak Galerkin method and
the nonconforming FE method can be seen by comparing the entries of the coefficient matrix $\bar{A}$.
The difference between the weak Galerkin method and the mixed-hybrid FE method, on the other hand,
can be observed from the following fact.
In the weak Galerkin method, the weak gradient $\nabla_w u$ lies in the discrete Raviart-Thomas space $\Sigma_h$
but the flux $\cA\nabla_w u$ does not, whereas in the mixed-hybrid FE, the flux $\cA\nabla u$ lies in the 
Raviart-Thomas space but the gradient $\nabla u$ does not.
\end{remark}

\begin{lemma}
\label{lem:spd}
Matrix $A$ is symmetric and positive definite.
\end{lemma}

\begin{proof}
It is easy to see that $A$ is symmetric. Next we show that $A$ is positive semi-definite. For any given vector
$\bv\in \mathbb{R}^{N_b}$, we denote $v_h = \sum_{i=1}^{N_b} v_i \phi_{b,i}$. Noticing $M_{0,b} = M_{b,0}^T$, we have
$$
\begin{aligned}
\bv^T A \bv &= \bv^T M_{b,b} \bv - (M_{0,b}\bv)^T M_{0,0}^{-1} (M_{0,b}\bv) \\
  &= (\cA \nabla_w v_h,\, \nabla_w v_h) -
     \sum_{i=1}^{N_0} \frac{(\cA \nabla_w v_h,\, \nabla_w \phi_{0,i})_{K_i}^2}{(\cA \nabla_w \phi_{0,i},\, \nabla_w \phi_{0,i})_{K_i}} \\
  &= \sum_{i=1}^{N_0} \left( (\cA \nabla_w v_h,\, \nabla_w v_h)_{K_i} -
      \frac{(\cA \nabla_w v_h,\, \nabla_w \phi_{0,i})_{K_i}^2}{(\cA \nabla_w \phi_{0,i},\, \nabla_w \phi_{0,i})_{K_i}} \right).
\end{aligned}
$$
Using the Schwartz inequality on each $K_i$, it is not difficult to see that $\bv^T A \bv \ge 0$.
To show $A$ is nonsingular, we notice that matrix $\bar{A} = \begin{bmatrix}A&A^{\partial}\\ 0 & I\end{bmatrix}$
comes from the
Schur complement of matrix $M$. Since $M$ is non-singular, its Schur complement must also be non-singular.
Hence, $A$ is non-singular.
\end{proof}

\begin{lemma}
\label{lem:rowsum}
All row sums of $\bar{A}$ are nonnegative.
\end{lemma}

\begin{proof}
Let $e_i$ be an interior edge. For each triangle $K\in \T_h$,
denote by $(x_i, y_i)$, $i=1,2,3$, the vertices of $K$ and by $\tilde{e}_i$ ($i=1,2,3$)
the locally indexed edge opposite to vertex $(x_i, y_i)$.
Also denote the unit outward normal vector on these three edges
by $\bn_1$, $\bn_2$ and $\bn_3$.
Let $\T_i$ be the collection of two triangles sharing edge $e_i$.
By Lemma \ref{lem:Aelem}, the sum of all entries on the $i^{\text{th}}$ row, $1\le i\le N_b$, of matrix $\bar{A}$ is
$$
\begin{aligned}
\sum_{j = 1}^{N_b+N_b^{\partial}} \bar{A}(i,j) 
& = \sum_{K\in \T_i}\frac{|e_i|}{|K|^2} \bigg( \left(\cA(|\tilde{e}_1|\bn_1 + |\tilde{e}_2|\bn_2 + |\tilde{e}_3|\bn_3),\,
   \bn_{i,K}\right)_K\\
& \qquad  \qquad \qquad - \frac{(\cA(\bx-\bx_K),\, \bn_{i,K})_K\, (\cA(\bx-\bx_K),\,
|\tilde{e}_1|\bn_1 + |\tilde{e}_2|\bn_2 + |\tilde{e}_3|\bn_3)_K}{\|\bx-\bx_K\|_{\cA,K}^2}\bigg).
\end{aligned}
$$
Notice that
\begin{equation} \label{eq:sumen}
\begin{aligned}
&|\tilde{e}_1|\bn_1 + |\tilde{e}_2|\bn_2 + |\tilde{e}_3|\bn_3 
= \begin{bmatrix}y_3-y_2\\ -(x_3-x_2)\end{bmatrix} + \begin{bmatrix}y_1-y_3\\ -(x_1-x_3)\end{bmatrix}
+ \begin{bmatrix}y_2-y_1\\ -(x_2-x_1)\end{bmatrix} 
= \mathbf{0}.
\end{aligned}
\end{equation}
Combining the above results, we know that the sum of each of the first $N_b$ rows of matrix $\bar{A}$ is $0$.
The rest of the row sums are just equal to $1$.
\end{proof}

\section{Discrete maximum principle}
\label{sec:DMP}

We now study the maximum principle for the weak Galerkin approximation (\ref{eq:SchurComplement}).
The weak Galerkin approximation to the solution of BVP (\ref{eq:pde}) on edges is said to satisfy DMP if
\begin{equation}
f(\bx) \le 0, \quad \forall \bx \in \Omega \quad\Longrightarrow \quad \max_{1\le i\le N_b} u_{b,i} \le \max\{0, \, \max_{1\le i\le N_b^{\partial}} u_{b,i}^{\partial} \}.
\label{eq:dmp-1}
\end{equation}
The maximum principle has been studied extensively in the past for systems in the form (\ref{eq:SchurComplement}).
For example, Ciarlet \cite{Cia70} shows that the DMP
\begin{equation}
-M_{b,0}M_{0,0}^{-1} \bF_0 \le 0
\quad\Longrightarrow \quad \max_{1\le i\le N_b} u_{b,i} \le \max\{0, \, \max_{1\le i\le N_b^{\partial}} u_{b,i}^{\partial} \}
\label{eq:dmp-2}
\end{equation}
holds if and only if
\begin{itemize}
\item[(a)] $\bar{A}$ is monotone, i.e., $\bar{A}$ is nonsingular and $\bar{A}^{-1}\ge 0$; and
\item[(b)] $\bxi + A^{-1}A^{\partial}\bxi^{\partial} \ge 0$, where $\bxi\in \mathbb{R}^{N_b}$ and
 $\bxi^{\partial}\in\mathbb{R}^{N_b^{\partial}}$ are vectors consisting of all ones,
\end{itemize}
where, and hereafter, unless stated otherwise the sign ``$\le$'' or ``$\ge$'' is in the elementwise sense
when used for vectors or matrices. 

The following Lemma is well-known. For completeness, a brief proof is provided.

\begin{lemma}
\label{lam:linalgDMP}
The above conditions (a) and (b) hold if
\begin{enumerate}
\item[(i)] $A$ is positive definite; and
\item[(ii)] All of the off-diagonal entries of $\bar{A}$ are nonpositive; and
\item[(iii)] All of the row sums of $\bar{A}$ are nonnegative.
\end{enumerate}
\end{lemma}

\begin{proof}
Conditions (ii) and (iii) imply that $\bar{A}$ is a Z-matrix (defined as a matrix with nonpositive off-diagonal entries
and nonnegative diagonal entries) and therefore, $A$ is a Z-matrix too.
This, together with (i), implies that $A$ is an $M$-matrix and thus $A^{-1} \ge 0$. 
Condition (a) follows by directly examining $\bar{A}^{-1} = \begin{bmatrix} A^{-1}  & - A^{-1} A^{\partial} \\
0 & I \end{bmatrix}$ and using (ii). Condition (b) follows from (iii) and the fact that $A$ is monotone.
\end{proof}

We should point out that there is a difference between (\ref{eq:dmp-1}) and (\ref{eq:dmp-2}).
Generally speaking, $f(\bx)\le 0$ does not guarantee 
\begin{equation}
-M_{b,0}M_{0,0}^{-1} \bF_0 \le 0 .
\label{eq:F-1}
\end{equation}
Thus, we need to include (\ref{eq:F-1}) as a part of the condition for the weak Galerkin approximation
to satisfy DMP.

We now examine system (\ref{eq:SchurComplement}) more closely.
From Lemmas \ref{lem:spd}, \ref{lem:rowsum}, and \ref{lam:linalgDMP},
to verify the maximum principle we need to check
the sign of the off-diagonal entries of $\bar{A}$ and the condition (\ref{eq:F-1}).
The off-diagonal entries of $\bar{A}$ are given in Lemma \ref{lem:Aelem}.
When $i\neq j$, we know that either $\T_{i,j}$ is empty, in which case $A(i,j) = 0$,
or $\T_{i,j}$ contains the only triangle $K\in \T_h$ that has both $e_i$ and $e_j$ as edges.
In this case, we have
\begin{equation}
\label{eq:Aij}
A(i,j) = \frac{|e_i|\, |e_j|}{|K|^2} \left[ (\cA\bn_{j,K},\, \bn_{i,K})_K -
\frac{(\cA(\bx-\bx_K),\, \bn_{i,K})_K\, (\cA(\bx-\bx_K),\, \bn_{j,K})_K}{\|\bx-\bx_K\|_{\cA,K}^2} \right] .
\end{equation}
Similarly, when interior edge $e_i$ and boundary edge $e_j^{\partial}$ are edges of triangle $K$,
\begin{equation}
\label{eq:Aij-2}
A^{\partial}(i,j) = \frac{|e_i|\, |e_j^{\partial}|}{|K|^2} \left[ (\cA\bn_{j,K}^{\partial},\, \bn_{i,K})_K -
\frac{(\cA(\bx-\bx_K),\, \bn_{i,K})_K\, (\cA(\bx-\bx_K),\, \bn_{j,K}^{\partial})_K}{\|\bx-\bx_K\|_{\cA,K}^2} \right] .
\end{equation}

\begin{theorem}
\label{thm:dmp}
If the mesh satisfies
\begin{align}
& (\cA\bn_{i,K},\, \bn_{j,K})_K \le
\frac{(\cA(\bx-\bx_K),\, \bn_{i,K})_K\, (\cA(\bx-\bx_K),\, \bn_{j,K})_K}{\|\bx-\bx_K\|_{\cA,K}^2},
\notag \\
& \qquad \qquad \qquad \qquad \qquad \qquad  \qquad \qquad \qquad \qquad
\forall\; K\in \T_h, \quad e_i, e_j, \in \partial K, \quad e_i \neq e_j
\label{eq:dmp-conditionMbb-2}
\\
& |(\cA(\bx-\bx_K),\, \bn_{i,K})_K | \le \frac{C_K |K|}{3|e_i|} \| \bx-\bx_K\|_{\cA,K}^2,\quad
\forall\; K\in \T_h, \quad e_i \in \partial K
\label{eq:dmp-conditionMbb}
\end{align}
then, $u_b$, the weak Galerkin approximation (\ref{eq:SchurComplement}) to the solution of BVP (\ref{eq:pde}) on edges,
satisfies the discrete maximum principle (\ref{eq:dmp-1}). 
\end{theorem}

\begin{proof}
From (\ref{eq:Aij}) and (\ref{eq:Aij-2}), (\ref{eq:dmp-conditionMbb-2}) implies that
all off-diagonal entries of matrix $\bar{A}$ are nonpositive.
Combining this with Lemmas \ref{lem:spd} and \ref{lem:rowsum}, we know that
the conditions in Lemma \ref{lam:linalgDMP} are satisfied.

For the condition (\ref{eq:F-1}), from Lemma \ref{lem:Mb0} we see that (\ref{eq:dmp-conditionMbb}) implies that
the entries of $M_{b,0}$  are all nonpositive. Moreover, 
from Lemma \ref{lem:M00}, we know that $M_{0,0}^{-1}$ is a diagonal matrix with positive diagonal entries.
From the definitions of $\phi_{0,i}$ and $\bF_0$, we have $\bF_0\le 0$ 
when $f(\bx) \le 0$. Thus, (\ref{eq:dmp-conditionMbb}) implies (\ref{eq:F-1})
and the solution of (\ref{eq:SchurComplement}) satisfies the DMP (\ref{eq:dmp-1}).
\end{proof}

\begin{theorem}
\label{thm:dmp-u0}
Under the assumptions of Theorem~\ref{thm:dmp}, $u_0$ satisfies the DMP
\begin{equation}
f(\bx) \le 0, \quad \forall \; \bx \in \Omega \quad\Longrightarrow \quad \max_{1\le i\le N_0} u_{0,i} \le
\max\{0, \, \max_{1\le i\le N_b^{\partial}} u_{b,i}^{\partial} \}.
\label{eq:dmp-3}
\end{equation}
Therefore, the weak Galerkin approximation (\ref{eq:uh}) satisfies the DMP
\begin{equation}
f(\bx) \le 0, \quad \forall\; \bx \in \Omega \quad\Longrightarrow \quad
\max\limits_{\stackrel{\bx \in \Omega}{\bx \text{ is not a vertex}} } u_h(\bx) 
\le \max\{0, \, \max_{1\le i\le N_b^{\partial}} u_{b,i}^{\partial} \} ,
\label{eq:dmp-4}
\end{equation}
where the values of $u_h(\bx)$ on vertices are excluded because 
$u_h(\bx)$ assumes multiple values on each vertex due to the discontinuity nature of the weak Galerkin approximation.
\end{theorem}

\begin{proof}
From  (\ref{eq:matrixform}), we have
\[
\bu_{0} = M_{0,0}^{-1} \left (\bF_0 - M_{0,b} \bu_b
- M_{0,b}^{\partial} \bu_b^{\partial}\right ) .
\]
From Lemma~\ref{lem:Mb0}, (\ref{eq:dmp-conditionMbb}) implies $M_{0,b}(i,j) \le 0$ and
$M_{0,b}^{\partial}(i,j) \le 0$. Moreover, $f(\bx) \le 0$ means $F_{0, i} \le 0$. Thus, letting $K$ be the $i^{\text{th}}$
element, from Theorem~\ref{thm:dmp} we have
\begin{align*}
u_{0, i} & = \frac{1}{M_{0,0}(i,i)} \left ( F_{0, i} - \sum_{e_j \in \partial K} M_{0,b}(i, j) u_{b, j} 
 - \sum_{e_j^{\partial} \in \partial K} M_{0,b}^{\partial}(i, j) u_{b, j}^{\partial} \right ) \\
&\le \frac{1}{M_{0,0}(i,i)} \left ( \sum_{e_j \in \partial K} (-M_{0,b}(i, j)) \cdot \max_{1\le k \le N_b} u_{b, k} 
 + \sum_{e_j^{\partial} \in \partial K} (- M_{0,b}^{\partial}(i, j)) \cdot \max_{1\le k \le N_b^{\partial}} u_{b, k}^{\partial} \right ) \\
 & \le \frac{1}{M_{0,0}(i,i)} \left ( \sum_{e_j \in \partial K} (- M_{0,b}(i, j) )
 + \sum_{e_j^{\partial} \in \partial K} (- M_{0,b}^{\partial}(i, j) ) \right ) 
 \max\{0, \, \max_{1\le k\le N_b^{\partial}} u_{b,k}^{\partial} \} .
\end{align*}
From Lemmas \ref{lem:M00} and \ref{lem:Mb0} and the identity (\ref{eq:sumen}), we get
\begin{align*}
u_{0, i} & \le \frac{\max\{0, \, \max_{1\le k\le N_b^{\partial}} u_{b,k}^{\partial} \} }
{C_K \|\bx - \bx_K\|_{\cA,K}^2} \sum_{e_j \in \partial K} \left (
\frac{1}{3}C_K \|\bx - \bx_K\|_{\cA,K}^2
+   \frac{|e_j|}{|K|} \left(\cA (\bx - \bx_K),\, \bn_{j,K} \right)_K \right )
\\ 
& = \frac{\max\{0, \, \max_{1\le k\le N_b^{\partial}} u_{b,k}^{\partial} \} }
{C_K \|\bx - \bx_K\|_{\cA,K}^2} \left ( C_K \|\bx - \bx_K\|_{\cA,K}^2
+   \frac{1}{|K|} (\cA (\bx - \bx_K),\, \sum_{e_j \in \partial K} |e_j| \bn_{j,K} )_K \right )
\\
& = \frac{\max\{0, \, \max_{1\le k\le N_b^{\partial}} u_{b,k}^{\partial} \} }
{C_K \|\bx - \bx_K\|_{\cA,K}^2} \left ( C_K \|\bx - \bx_K\|_{\cA,K}^2
+   0\frac{}{} \right ) = \max\{0, \, \max_{1\le k\le N_b^{\partial}} u_{b,k}^{\partial} \},
\end{align*}
which implies (\ref{eq:dmp-3}).
Combining this with Theorem~\ref{thm:dmp} gives (\ref{eq:dmp-4}).
\end{proof}

\begin{remark} \label{rem:numericalintegration}
In actual computation, the $L^2$ inner-products on a triangle $K$ involved in the weak Galerkin approximation
(\ref{eq:SchurComplement}) are typically calculated using quadrature rules.
Since most of these quadrature rules still define an inner-product on polynomial functions, the
above analysis as well as those to be given below in  \S\ref{sec:pwconstant} and \S\ref{sec:general}
can be extended to the situation with numerical integration.
In this case, we need to replace the $L^2$ inner-products in the analysis by the discrete $L^2$ inner-product
associated with the quadrature rule and to require that the discrete inner-product satisfy
condition (\ref{eq:basis0}), which is true as long as the quadrature is exact for linear polynomials.
\qed
\end{remark}

Next we look into the conditions in Theorem \ref{thm:dmp} in more detail.
Let $\cA_K$ be the average of $\cA$ over $K$, i.e.,
\[
\cA_K = \frac{1}{|K|} \int_K \cA \,d\bx .
\]
Then we can rewrite the left-hand side of (\ref{eq:dmp-conditionMbb-2}) as
\[
(\cA\bn_{i,K},\, \bn_{j,K})_K
= |K|\;  \bn_{j,K}^T \cA_K \bn_{i,K} .
\]
Denote the unit directions (with the vertices of $K$ being ordered counterclockwisely)
along edges $e_i$ and $e_j$ by $\be_i$ and $\be_j$, respectively. By direct calculation one has
\[
\bn_{j,K}^T \cA_K \bn_{i,K} = \det(\cA_K) \; \be_j^T \cA_K^{-1} \be_i .
\]
Denote by $\alpha_{i,j,\cA_K^{-1}}$ the angle (in $K$) formed by $e_i$ and $e_j$
and measured in the metric specified by $\cA_K^{-1}$. By definition, we have
\[
\cos\left ( \alpha_{i,j,\cA_K^{-1}}\right ) = - \frac{\be_j^T \cA_K^{-1} \be_i}{\sqrt{\be_i^T \cA_K^{-1} \be_i} \cdot 
\sqrt{\be_j^T \cA_K^{-1} \be_j}} .
\]
Combining the above results, we get
\begin{align}
(\cA\bn_{i,K},\, \bn_{j,K})_K
= - |K|\, \det(\cA_K)\, \cos\left ( \alpha_{i,j,\cA_K^{-1}}\right )  
\sqrt{\be_i^T \cA_K^{-1} \be_i} \cdot  \sqrt{\be_j^T \cA_K^{-1} \be_j} .
\label{eq:alpha}
\end{align}
From this, we can see that the statements that $(\cA\bn_{i,K},\, \bn_{j,K})_K \le 0$ and
the angle $\alpha_{i,j,\cA_K^{-1}}$ is nonobtuse are equivalent.

It is noted that the conditions (\ref{eq:dmp-conditionMbb-2}) and (\ref{eq:dmp-conditionMbb})
can be simplified significantly when $\cA$ is piecewise constant on $\T_h$.
For this reason, we study this situation first in the following and discuss the general situation afterward.

\subsection{The case with piecewise constant $\cA$}
\label{sec:pwconstant}

For this case, from (\ref{eq:basis0}) the conditions (\ref{eq:dmp-conditionMbb-2}) and (\ref{eq:dmp-conditionMbb})
reduce to
\begin{align*}
& (\cA\bn_{i,K},\, \bn_{j,K})_K \le
\frac{(\cA(\bx-\bx_K),\, \bn_{i,K})_K\, (\cA(\bx-\bx_K),\, \bn_{j,K})_K}{\|\bx-\bx_K\|_{\cA,K}^2} = 0,
\notag \\
& \qquad \qquad \qquad \qquad \qquad \qquad  \qquad \qquad \qquad \qquad
\forall\; K\in \T_h, \quad e_i, e_j \in \partial K, \quad e_i \neq e_j,
\\
& |(\cA(\bx-\bx_K),\, \bn_{i,K})_K | = 0 \le \frac{C_K |K|}{3|e_i|} \| \bx-\bx_K\|_{\cA,K}^2,\quad
\forall\; K\in \T_h, \quad e_i \in \partial K.
\end{align*}
Thus, (\ref{eq:dmp-conditionMbb}) is satisfied automatically. Moreover, from (\ref{eq:alpha}) one can see
that (\ref{eq:dmp-conditionMbb-2}) is true if all the angles of the mesh are nonobtuse when measured
in the metric specified by $\cA_K^{-1}$. Combining this with Theorem~\ref{thm:dmp-u0}, we have the following
theorem.

\begin{theorem}
\label{thm:pwconstant}
If $\cA$ is piecewise constant on $\T_h$ and all of the mesh angles are nonobtuse when measured
in the metric specified by $\cA_K^{-1}$, the weak Galerkin approximation defined in (\ref{eq:wg}) and
(\ref{eq:uh}) satisfies the DMP (\ref{eq:dmp-4}).
\end{theorem}

\begin{remark}
The mesh condition in Theorem \ref{thm:pwconstant} is referred to as the anisotropic
nonobtuse angle condition by Li and Huang \cite{LH10}, a generalization of the well-known
nonobtuse angle condition \cite{BKK08, CR73} to the case with a general anisotropic diffusion matrix.
They show that the P1 conforming FE approximation
to BVP (\ref{eq:pde}) satisfies a DMP when the mesh condition holds.
Like the isotropic diffusion case \cite{Let92, SF73, XZ99}, it is also shown in \cite{Hua11}
that the condition can be replaced by a weaker, Delaunay-type mesh condition in two dimensions.
Unfortunately, this may not be true for the weak Galerkin approximation.
This is because in the P1 conforming FE approximation, basis functions are associated with
vertices and the support of basis functions associated with any pair of neighboring vertices
can overlap over two triangles. It is this two-triangle overlap that leads to a weaker condition in two dimensions.
On the other hand, the system (\ref{eq:SchurComplement}) involves basis functions associated with edges
and the support of basis functions based on any pair of neighboring edges
overlaps over at most a triangle, which unlikely leads to a weaker mesh condition.
\qed
\end{remark}


\subsection{The case with a general anisotropic matrix $\cA$}
\label{sec:general}

The general case is considered as a perturbation of the piecewise constant case. Define
\[
\lambda_{\min, K}(\cA) = \min\limits_{\bx \in K} \lambda_{\min} (\cA(\bx)),
\]
where $\lambda_{\min} (\cA(\bx))$ denotes the minimal eigenvalue of $\cA(\bx)$.
We assume that $\cA$ is Lipschitz continuous on each element, i.e., for any $K \in \T_h$, there exists
a constant $L_{\cA, K}$ such that
\[
|\cA(\bx) - \cA(\by)| \le L_{\cA, K} |\bx - \by|, \quad \forall \, \bx, \by \in K .
\]
Then, by the mean value theorem we have
$$
|\cA(\bx) - \cA_K| \le L_{\cA, K} h_K, \qquad \forall\,  \bx \in K.
$$

\begin{theorem}
Assume that $\cA$ is Lipschitz continuous on each element of $\T_h$. If the mesh satisfies
\begin{align}
&  \frac{L_{\cA, K}^2 h_K^2}{\lambda_{\min, K}^2(\cA)} \le \cos\left ( \alpha_{i,j,\cA_K^{-1}} \right ),
\quad \forall e_i, e_j \in \partial K,\quad e_i \ne e_j, \quad \forall K \in \T_h
\label{eq:thm:general-1}
\\
& \frac{h_K^3}{|K|} \le \frac{2 \lambda_{\min, K}(\cA)}{3 L_{\cA, K}},
\quad \forall K \in \T_h
\label{eq:thm:general-2}
\end{align}
then the weak Galerkin approximation defined in (\ref{eq:wg}) and
(\ref{eq:uh}) satisfies the DMP (\ref{eq:dmp-4}).
\label{thm:general}
\end{theorem}

\begin{proof}
We first consider the condition (\ref{eq:dmp-conditionMbb-2}). Notice that $\det(\cA_K) = \lambda_{\max}(\cA_K)
\lambda_{\min}(\cA_K)$ and
\[
\be_i^T \cA_K^{-1} \be_i \ge \frac{\be_i^T \be_i}{\lambda_{\max}(\cA_K)} = \frac{1}{\lambda_{\max}(\cA_K)} ,
\qquad \be_j^T \cA_K^{-1} \be_j \ge \frac{1}{\lambda_{\max}(\cA_K)} .
\]
Assuming that $\alpha_{i,j,\cA_K^{-1}}$ is nonobtuse, from (\ref{eq:alpha}) we have
\begin{align}
(\cA\bn_{i,K},\, \bn_{j,K})_K
& = - |K|\, \lambda_{\max}(\cA_K) \lambda_{\min}(\cA_K)\, \cos\left ( \alpha_{i,j,\cA_K^{-1}}\right )  
\sqrt{\be_i^T \cA_K^{-1} \be_i} \cdot  \sqrt{\be_j^T \cA_K^{-1} \be_j} .
\notag \\
& \le - \frac{|K|\, \lambda_{\max}(\cA_K) \lambda_{\min}(\cA_K)\, \cos\left ( \alpha_{i,j,\cA_K^{-1}}\right )  }
{\lambda_{\max}(\cA_K) }
\notag \\
& \le - |K|\, \lambda_{\min, K}(\cA) \, \cos\left ( \alpha_{i,j,\cA_K^{-1}}\right )  .
\label{eq:thm:general-3}
\end{align}
For the right-hand side of (\ref{eq:dmp-conditionMbb-2}), we have
\begin{align}
& \frac{(\cA(\bx-\bx_K),\, \bn_{i,K})_K\, (\cA(\bx-\bx_K),\, \bn_{j,K})_K}{\|\bx-\bx_K\|_{\cA,K}^2}
\notag \\
& = \frac{((\cA-\cA_K) (\bx-\bx_K),\, \bn_{i,K})_K\, ( (\cA-\cA_K)(\bx-\bx_K),\, \bn_{j,K})_K}{\|\bx-\bx_K\|_{\cA,K}^2}
\notag \\
& \ge - \frac{|((\cA-\cA_K) (\bx-\bx_K),\, \bn_{i,K})_K\, ( (\cA-\cA_K)(\bx-\bx_K),\, \bn_{j,K})_K|}{\|\bx-\bx_K\|_{\cA,K}^2}
\notag \\
&\ge - \frac{ \left(\int_K |(\cA-\cA_K) (\bx-\bx_K)|\, d\bx \right)^2 }{\|\bx-\bx_K\|_{\cA,  K}^2}
\notag \\
&\ge - \frac{ L_{\cA, K}^2 h_K^2 \left(\int_K |\bx-\bx_K|\, d\bx\right)^2 }{\lambda_{\min, K}(\cA) \|\bx-\bx_K\|_{K}^2}
\notag \\
&\ge - \frac{ L_{\cA, K}^2 |K|\, h_K^2}{\lambda_{\min, K}(\cA)} . 
\notag
\end{align}
From this and (\ref{eq:thm:general-3}), we know that (\ref{eq:dmp-conditionMbb-2}) is true
when (\ref{eq:thm:general-1}) holds.

We now consider the condition (\ref{eq:dmp-conditionMbb}). For the left-hand side, we have
\begin{align}
|(\cA(\bx-\bx_K),\, \bn_{i,K})_K | 
 & = |((\cA-\cA_K) (\bx-\bx_K),\, \bn_{i,K})_K |
 \notag \\
 & \le L_{\cA,K} h_K \int_K |\bx-\bx_K|\, d\bx 
 \notag \\
& \le L_{\cA,K} h_K^2 |K| .
\label{eq:thm:general-4}
\end{align}
For the right-hand side, we get
\[
\frac{C_K |K|}{3|e_i|} \| \bx-\bx_K\|_{\cA,K}^2 \ge \frac{C_K |K|}{3|e_i|} \lambda_{\min, K}(\cA) \| \bx-\bx_K\|_{K}^2 
= \frac{2\lambda_{\min, K}(\cA) |K|^2}{3|e_i|} \ge \frac{2\lambda_{\min, K}(\cA) |K|^2}{3 h_K} .
\]
From this and (\ref{eq:thm:general-4}), we know that (\ref{eq:dmp-conditionMbb}) is true
when (\ref{eq:thm:general-2}) holds.

The conclusion is then drawn from Theorem~\ref{thm:dmp-u0}.
\end{proof}

\begin{remark}
When $\cA$ is piecewise constant on $\T_h$, we will have $L_{\cA, K} = 0$ for all $K \in \T_h$.
It is easy to see that Theorem~\ref{thm:general} reduces to Theorem~\ref{thm:pwconstant} in this case.
\qed
\end{remark}

\begin{remark}
The mesh condition (\ref{eq:thm:general-1}) requires that the mesh be $\mathcal{O}(h^2)$-acute, i.e.,
all of the mesh angles, measured in the metric specified by $\cA_K^{-1}$,
are $\mathcal{O}(h^2)$ away from being the right angle.
On the other hand, the mesh condition (\ref{eq:thm:general-2}) is less restrictive, which can be satisfied
as long as the mesh is sufficiently fine and the elements are not very skew.
\qed
\end{remark}

\section{Numerical Results}
\label{sec:numerical}

In this section we present some numerical results to illustrate the theoretical analysis in the previous sections.

\begin{exam}
\label{exam5.1}
The first test problem is in the form (\ref{eq:pde}) with $\Omega = (0,16)\times (0,16)$,
$$
\begin{aligned}
\cA &= \begin{bmatrix}500.5 & 499.5\\ 499.5 & 500.5\end{bmatrix} ,\quad
f = 0,\quad \textrm{ and } \quad g = \begin{cases}1, &\textrm{for } 0\le x\le 14,\, y=16 \\
                              8-0.5x, \quad &\textrm{for } 14<x<16,\, y=16 \\
                              1, &\textrm{for } x = 0,\, 2\le y\le 16 \\
                              0.5y, &\textrm{for }x = 0,\, 0<y<2 \\
                              0, &\textrm{otherwise}.
\end{cases} 
\end{aligned}
$$
This example has been studied in \cite{Hua11,LH10}. Notice that the diffusion coefficient matrix is constant on $\Omega$.
We solve this problem using the weak Galerkin method
on three types of mesh as shown in Fig.~\ref{fig:T1meshes}.
Among them, {\tt mesh45} and {\tt mesh90} satisfy the mesh conditions in Theorem~\ref{thm:pwconstant}
whereas {\tt mesh135} does not.
The maximum and minimum values of both ${u}_b$ and $u_0$
are reported in Table \ref{tab:T1maxmin}. These results confirm the theoretical 
predictions in Theorem~\ref{thm:pwconstant}: both $u_b$ and $u_0$ obtained with
{\tt mesh45} and {\tt mesh90} remain within the range between 0 and 1 but those obtained
with {\tt mesh135} have undershoots and overshoots. 
Contour plots, drawn using the average values at vertices, are shown in Fig.~\ref{fig:T1contour}. 
They are consistent with the above observation.
\qed
\end{exam}

\begin{figure}
  \begin{center}
    \includegraphics[width=5cm]{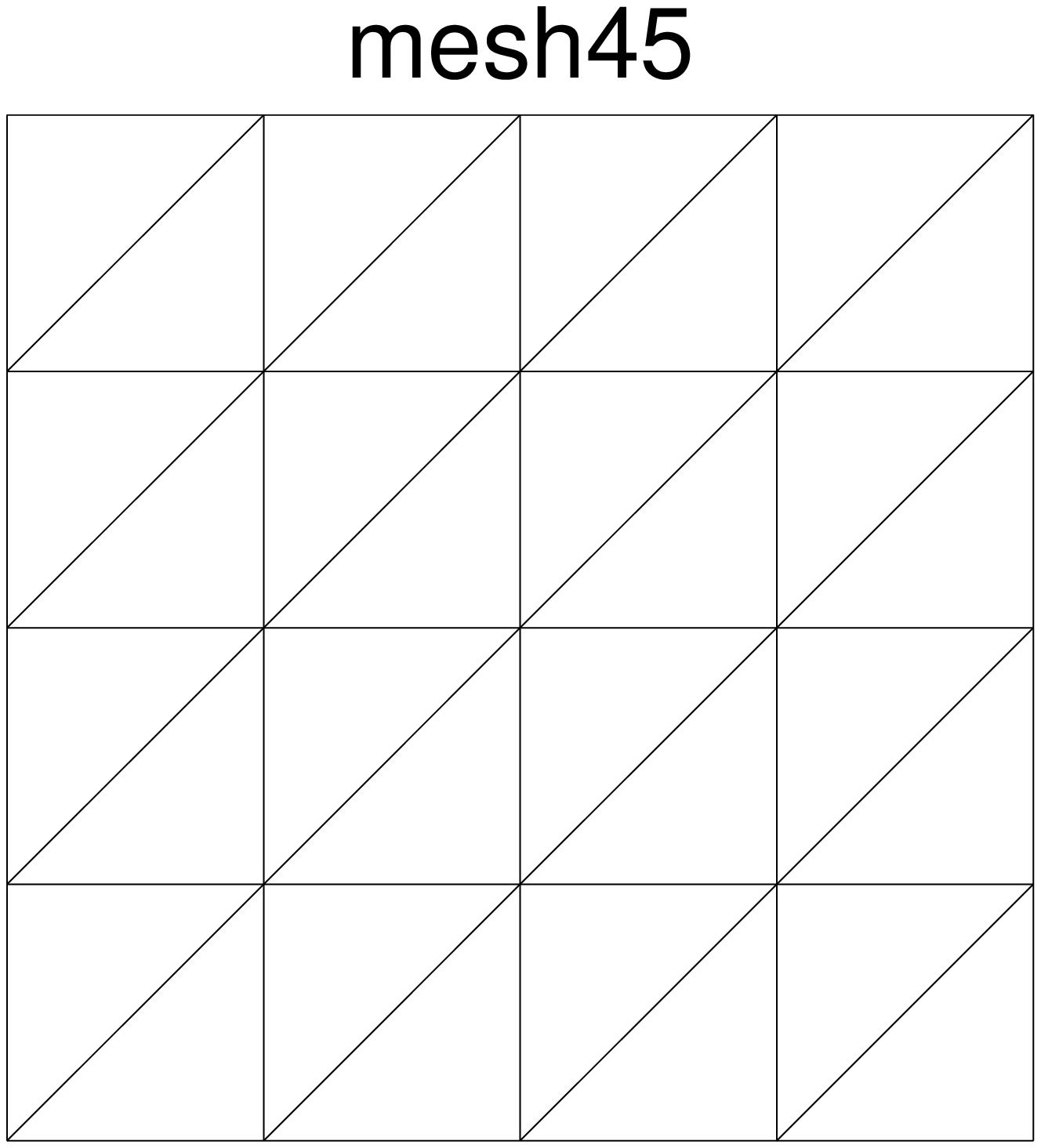} \hspace{-0.2cm}
    \includegraphics[width=5cm]{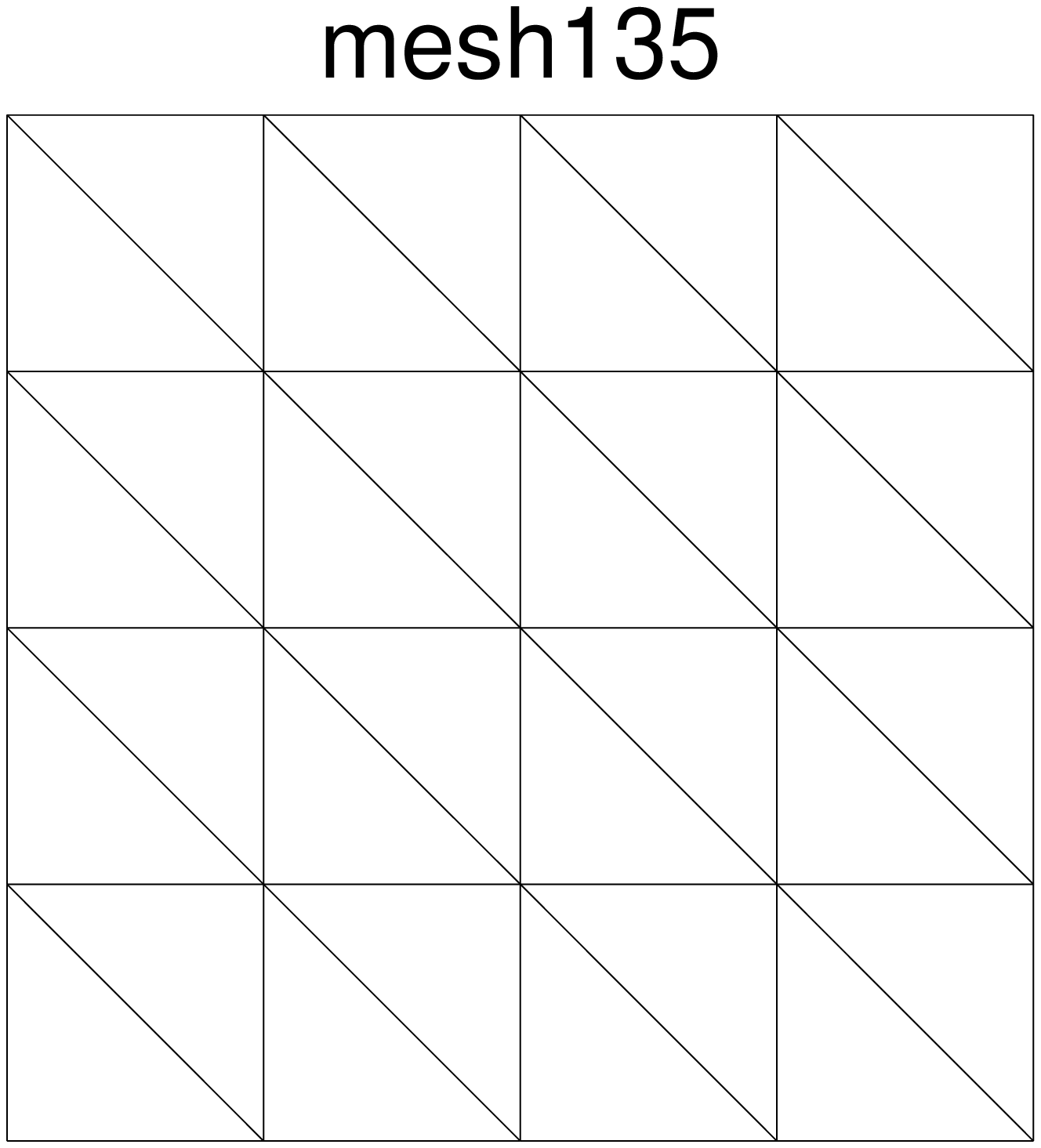} \hspace{-0.2cm}
    \includegraphics[width=5cm]{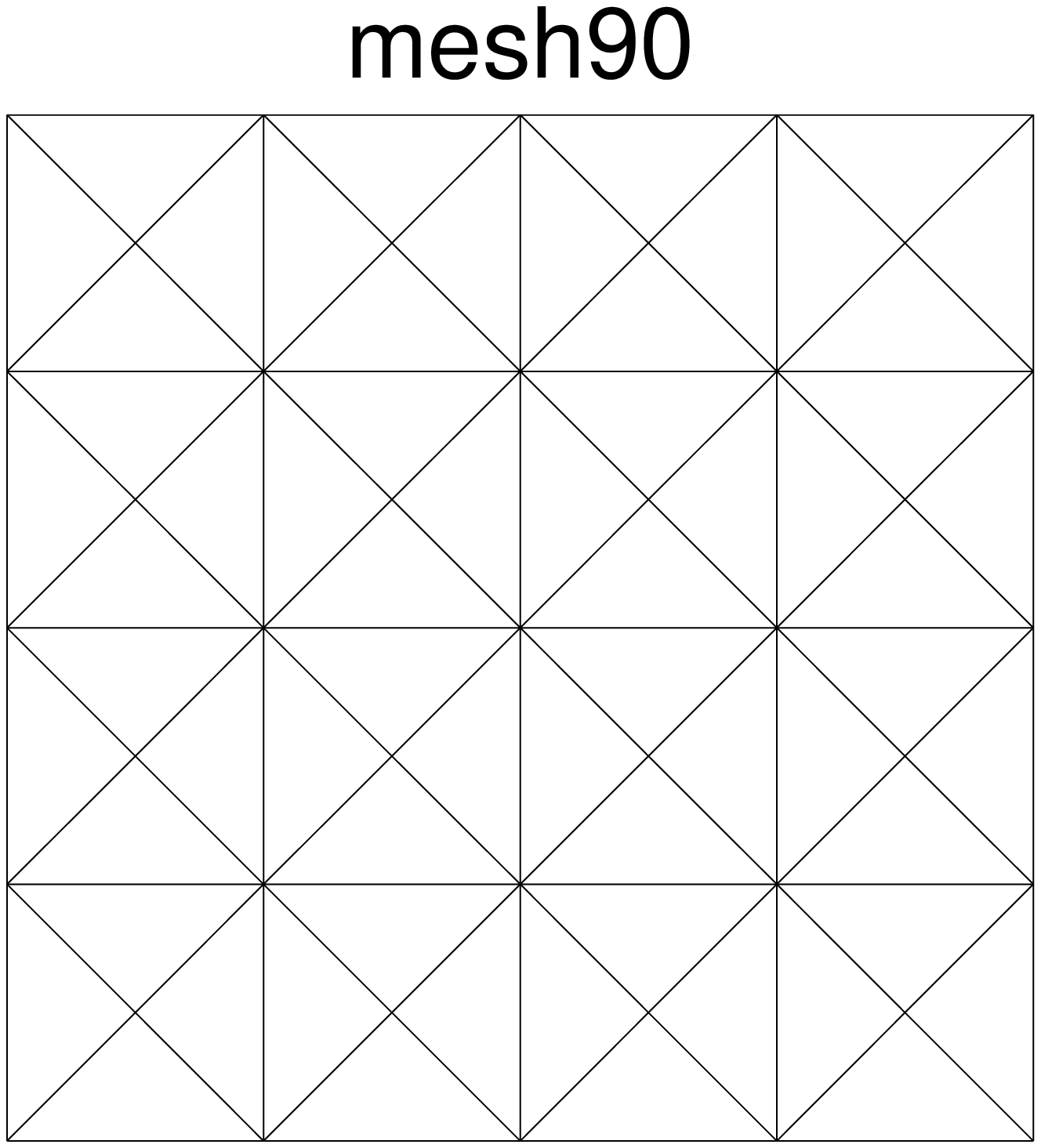} 
  \end{center}
\caption{Three types of mesh are used for Example~\ref{exam5.1}.}
\label{fig:T1meshes}
\end{figure}

\begin{table}
  \begin{center}
    \caption{Maximum and minimum values of the numerical solutions obtained with
    							different meshes for Example~\ref{exam5.1}.} \vspace{5pt}
    \label{tab:T1maxmin}
    \begin{tabular}{|c|c|c|c|c|c|c|c|c|c|c|c|c|}
      \hline
       & \multicolumn{4}{|c|}{\tt mesh45} & \multicolumn{4}{|c|}{\tt mesh135} & \multicolumn{4}{|c|}{\tt mesh90}  \\ \hline
      Size & \multicolumn{2}{|c|}{Max.} & \multicolumn{2}{|c|}{Min.} & \multicolumn{2}{|c|}{Max.} & \multicolumn{2}{|c|}{Min.}
                     & \multicolumn{2}{|c|}{Max.} & \multicolumn{2}{|c|}{Min.} \\ \hline
      & $u_b$ & $u_0$ & $u_b$ & $u_0$ & $u_b$ & $u_0$ & $u_b$ & $u_0$ & $u_b$ & $u_0$
      & $u_b$ & $u_0$  \\ \hline
      $8\times 8$   & $1$ & $1$ & $0$ & $0$ & $1.038$ & $1.019$ & $-5.14E-2$ & $-2.57E-2$ & $1$ & $1$ & $0$ & $0$ \\ \hline
      $16\times 16$ & $1$ & $1$ & $0$ & $0$ & $1.041$ & $1.026$ & $-5.01E-2$ & $-3.20E-2$ & $1$ & $1$ & $0$ & $0$ \\ \hline
      $32\times 32$ & $1$ & $1$ & $0$ & $0$ & $1.035$ & $1.028$ & $-4.05E-2$ & $-3.36E-2$ & $1$ & $1$ & $0$ & $0$ \\ \hline
      $64\times 64$ & $1$ & $1$ & $0$ & $0$ & $1.028$ & $1.027$ & $-3.19E-2$ & $-3.09E-2$ & $1$ & $1$ & $0$ & $0$ \\ \hline
    \end{tabular}
  \end{center}
\end{table}


\begin{exam}
\label{exam5.2}
To test the discrete maximum principle for non-constant diffusion coefficients, we consider an example
in the form (\ref{eq:pde}) with $\Omega = (0,1)\times (0,1)$. Denote by $(r,\theta)$ the polar coordinates
with the pole centered at $(x,y) =(-0.1,0.5)$. The diffusion matrix is defined as
$$
\cA = \begin{bmatrix}\cos(\theta) & \sin(\theta) \\ - \sin(\theta) & \cos(\theta)\end{bmatrix} 
 \begin{bmatrix}k_1 & 0 \\ 0 & k_2\end{bmatrix} 
 \begin{bmatrix}\cos(\theta) & - \sin(\theta) \\ \sin(\theta) & \cos(\theta) \end{bmatrix} ,
$$
where $k_1 = 1$, $k_2 = 1 + \gamma e^{ -200 (r-0.5)^2 }$, and $\gamma$ be a positive parameter. Notice that
$k_2$ is a Gaussian distribution in $r$ and peaks at $r=0.5$ with a maximum value $\gamma+1$. 
The diffusion matrix becomes more anisotropic around $r= 0.5$ for larger $\gamma$.
We choose $f = 0$ and $g = \sin (x+0.5)\pi$.
The maximum principle implies that the exact solution of the BVP stays between $-1$ and $1$.

\begin{figure}
  \begin{center}
    \includegraphics[width=5cm]{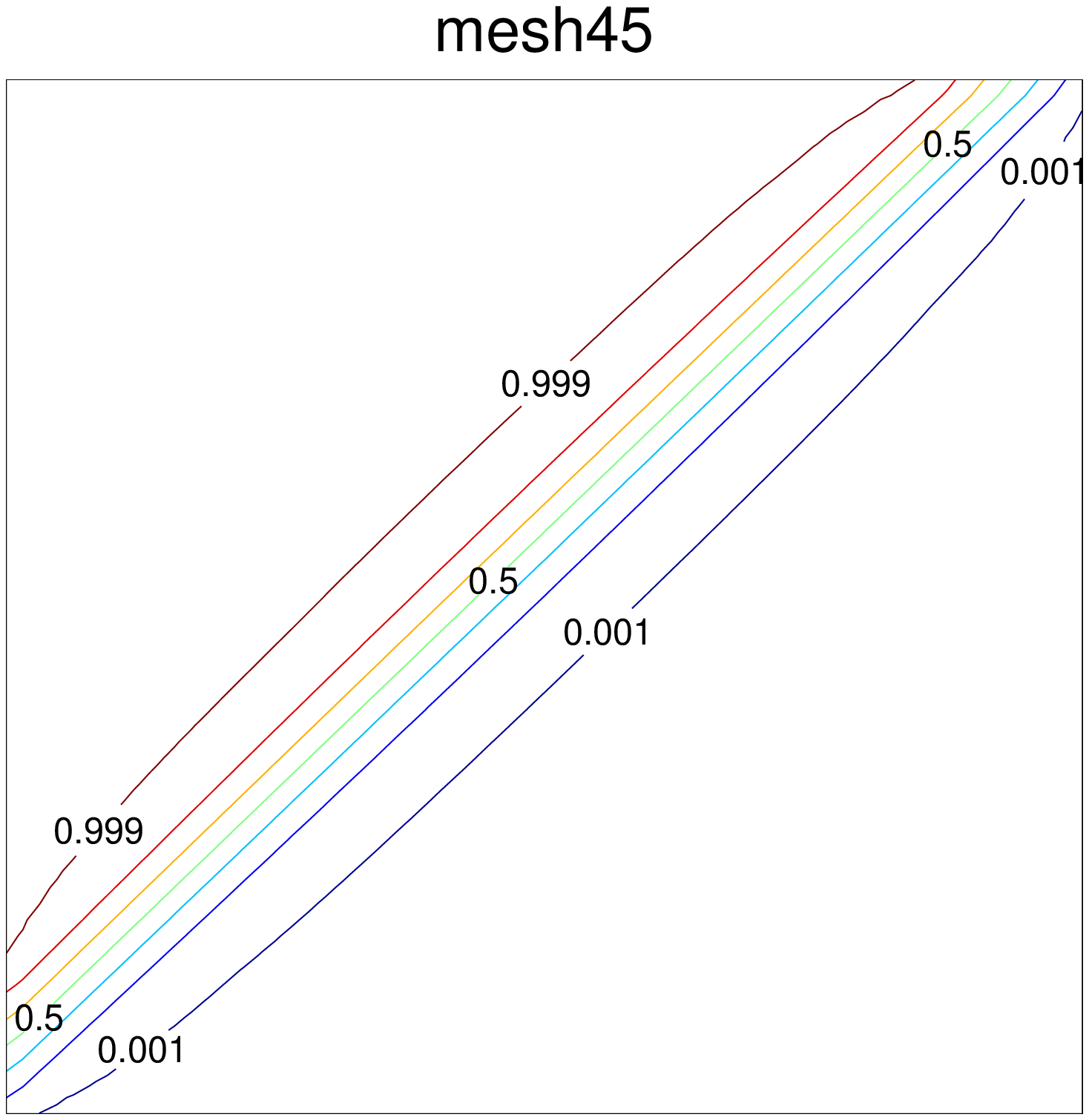} \hspace{-0.2cm}
    \includegraphics[width=5cm]{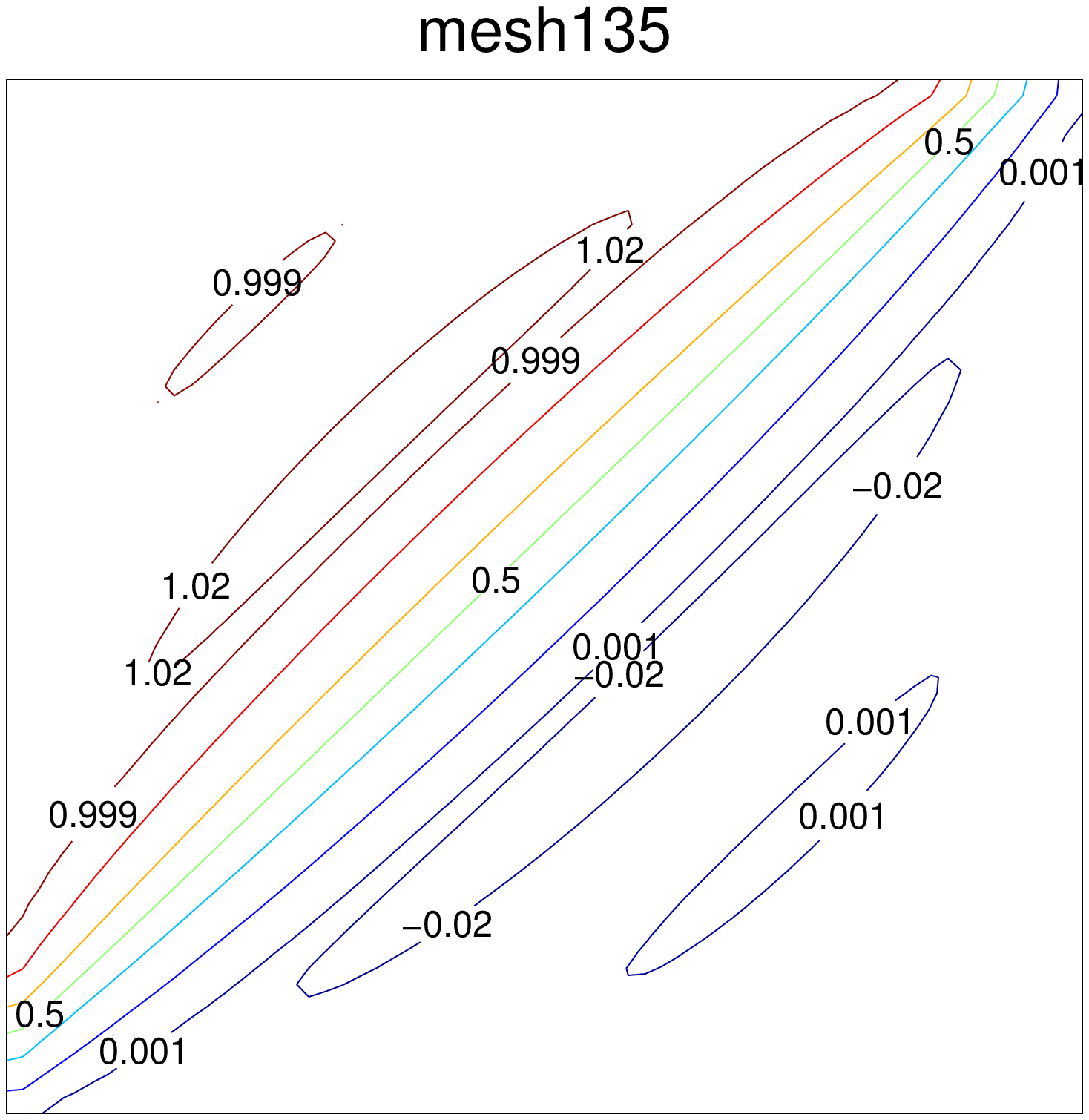} \hspace{-0.2cm}
    \includegraphics[width=5cm]{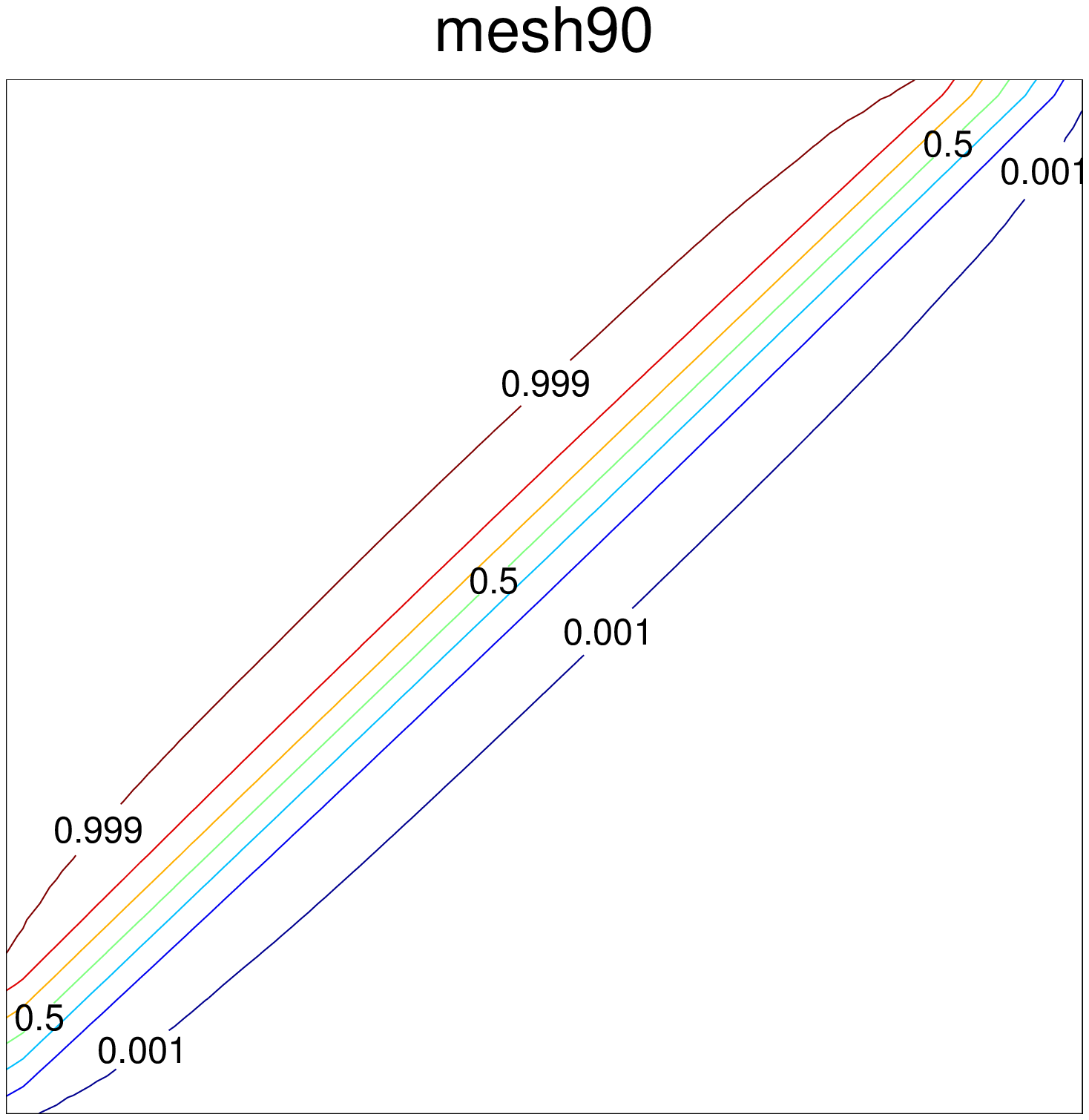} 
  \end{center}
\caption{Contours of the numerical solutions obtained with three types of mesh (with size $64\times 64$)
    for Example~\ref{exam5.1}.}
\label{fig:T1contour}
\end{figure}

We now test this problem on all three meshes in Fig. \ref{fig:T1meshes}. 
Because of symmetry, {\tt mesh45} and {\tt mesh135} give almost identical results up to a reflection across $y=0.5$,
thus we examine here only the maximum and minimum values on {\tt mesh45} and {\tt mesh90}. 
In Table \ref{tab:T2maxmin-mesh45}--\ref{tab:T2maxminu0-mesh90},
the maximum and minimum values of $u_0$ and $u_b$, computed on {\tt mesh45} and {\tt mesh90}
using various values of $\gamma$, are reported. 
Both $u_b$ and $u_0$ have overshoots on some meshes.
Moreover, we have marked the triangles and edges, on which $u_h$ has an overshoot,
in Figs. \ref{fig:T2overshoot-mesh45} and \ref{fig:T2overshoot-mesh90}, for $\gamma=99$ and various mesh sizes.
Behavior for other values of $\gamma$ and mesh sizes are similar and hence omitted.

It is interesting to point out that both {\tt mesh45} and {\tt mesh90} do not satisfy the mesh condition
(\ref{eq:thm:general-1}) for any value of $\gamma$ or any mesh size. To see this,
we first notice that the diffusion matrix $\cA$ has negative off-diagonal entries, $(k_2-k_1)\sin\theta\cos\theta$, 
at all points below the line $y=0.5$ and so does $\cA_K$ at all triangles below this line.
From (\ref{eq:alpha}), one can see that $\cos(\alpha_{i,j,\cA_K^{-1}})$ are negative 
when $e_i$ and $e_j$ are horizontal and vertical edges of those triangles. Thus, 
{\tt mesh45} does not satisfy (\ref{eq:thm:general-1}). For {\tt mesh90}, we consider the four triangles
within an arbitrary square and denote the unit normal of the diagonal lines by $\tilde{\bn}_1$ and $\tilde{\bn}_2$.
We assume that $h$ is sufficiently small so that $\cA_K$ is almost the same on these triangles.
Then, the left-hand side of (\ref{eq:alpha}) takes value $\tilde{\bn}_1^T \cA_K \tilde{\bn}_2$
on two of those triangles and $-\tilde{\bn}_1^T \cA_K \tilde{\bn}_2$ on the other. This means that
$\cos(\alpha_{i,j,\cA_K^{-1}})$ takes negative values on two of the triangles and therefore {\tt mesh90} violates
(\ref{eq:thm:general-1}).

The above analysis explains why the maximum principle is violated for most cases shown in
Tables \ref{tab:T2maxmin-mesh45}--\ref{tab:T2maxminu0-mesh90}. On the other hand, the tables
also show that the magnitudes of the undershoots and overshoots decrease as $h \to 0$,
which is consistent with the fact that the weak Galerkin approximation is convergent
\cite{WG-biharmonic,WangYe_PrepSINUM_2011}. Moreover, one can see from the tables that
the maximum principle is satisfied for some cases even when the mesh condition (\ref{eq:thm:general-1})
is violated. This does not contradict the theoretical analysis since (\ref{eq:thm:general-1}) is
only a sufficient condition.

Another observation from Table \ref{tab:T2maxmin-mesh45}--\ref{tab:T2maxminu0-mesh90} is that
increasing the value of $\gamma$ worsens the violation of the maximum principle. 
This is because the problem becomes more anisotropic when $\gamma$ gets larger.

Next, we solve the Example~\ref{exam5.2} on meshes generated by  
the Delaunay-type triangulator BAMG ({B}idimensional {A}nisotropic {M}esh {G}enerator,
developed by Hecht \cite{Hec97}).
BAMG is designed to generate triangular meshes with a given metric tensor. 
Meshes generated by BAMG with $\cA^{-1}$ as the metric tensor are shown in 
Figs.~\ref{fig:BAMGmesh20} and \ref{fig:BAMGmesh99} for different values of $\gamma$. 
These meshes match well with the diffusion coefficient $\cA$, which becomes more anisotropic
around $r=0.5$ and remains nearly isotropic away from $r=0.5$. By comparing meshes with different
values of $\gamma$, we can see that the triangles become more skewed around $r=0.5$
as $\gamma$ becomes larger.
Numerical experiments show that weak Galerkin solutions on these meshes satisfy the discrete maximum principle
(the results are not shown to save space): all entries of $u_b$ as well as $u_0$
lie between $-1$ and $1$, which agrees with the theoretical prediction given in Theorem \ref{thm:general}.
\qed
\end{exam}

\begin{table}
  \begin{center}
    \caption{Maximum and minimum values of $u_b$ obtained with
    		different $\gamma$ for Example~\ref{exam5.2} on {\tt mesh45}.} \vspace{5pt}
    \label{tab:T2maxmin-mesh45}
    \begin{tabular}{|c|c|c|c|c|c|c|c|c|c|c|c|c|}
      \hline
       & \multicolumn{2}{|c|}{$\gamma=20$}  & \multicolumn{2}{|c|}{$\gamma=40$} & \multicolumn{2}{|c|}{$\gamma=60$} & \multicolumn{2}{|c|}{$\gamma=99$}  \\ \hline
      Size & Max. & Min. & Max. & Min. & Max. & Min. & Max. & Min.  \\ \hline
      $8\times 8$   & $1.02$  & $-1$ & $1.038$ & $-1$ & $1.045$ & $-1$ & $1.051$ & $-1$  \\ \hline
      $16\times 16$ & $1.002$ & $-1$ & $1.012$ & $-1$ & $1.016$ & $-1$ & $1.019$ & $-1$ \\ \hline
      $32\times 32$ & $1$     & $-1$ & $1.002$ & $-1$ & $1.004$ & $-1$ & $1.006$ & $-1$ \\ \hline
      $64\times 64$ & $1$     & $-1$ & $1$     & $-1$ & $1.001$ & $-1$ & $1.002$ & $-1$ \\ \hline
    \end{tabular}
  \end{center}
\end{table}

\begin{table}
  \begin{center}
    \caption{Maximum and minimum values of $u_0$ obtained with
    		different $\gamma$ for Example~\ref{exam5.2} on {\tt mesh45}.} \vspace{5pt}
    \label{tab:T2u0maxmin-mesh45}
    \begin{tabular}{|c|c|c|c|c|c|c|c|c|c|c|c|c|}
      \hline
       & \multicolumn{2}{|c|}{$\gamma=20$}  & \multicolumn{2}{|c|}{$\gamma=40$} & \multicolumn{2}{|c|}{$\gamma=60$} & \multicolumn{2}{|c|}{$\gamma=99$}  \\ \hline
      Size & Max. & Min. & Max. & Min. & Max. & Min. & Max. & Min.  \\ \hline
      $8\times 8$   & $.992$ & $-.971$ & $1.004$& $-.971$ & $1.01$  & $-.971$ & $1.015$ & $-.971$  \\ \hline
      $16\times 16$ & $.996$ & $-.991$ & $.998$ & $-.991$ & $1.001$ & $-.991$ & $1.005$ & $-.991$ \\ \hline
      $32\times 32$ & $.998$ & $-.997$ & $.999$ & $-.997$ & $.999$  & $-.997$ & $1.001$ & $-.997$ \\ \hline
      $64\times 64$ & $.999$ & $-.999$ & $.999$ & $-.999$ & $.999$  & $-.999$ & $1.000$ & $-.999$ \\ \hline
    \end{tabular}
  \end{center}
\end{table}

\begin{table}
  \begin{center}
    \caption{Maximum and minimum values of $u_b$ obtained with
    		different $\gamma$ for Example~\ref{exam5.2} on {\tt mesh90}.} \vspace{5pt}
    \label{tab:T2maxmin-mesh90}
    \begin{tabular}{|c|c|c|c|c|c|c|c|c|c|c|c|c|}
      \hline
       & \multicolumn{2}{|c|}{$\gamma=20$}  & \multicolumn{2}{|c|}{$\gamma=40$} & \multicolumn{2}{|c|}{$\gamma=60$} & \multicolumn{2}{|c|}{$\gamma=99$}  \\ \hline
      Size & Max. & Min. & Max. & Min. & Max. & Min. & Max. & Min.  \\ \hline
      $8\times 8$   & $1.0098$ & $-1$ & $1.006$  & $-1$ & $1.002$ & $-1$ & $1.010$ & $-1$  \\ \hline
      $16\times 16$ & $1.003$  & $-1$ & $1.003$  & $-1$ & $1.002$ & $-1$ & $1.003$ & $-1$ \\ \hline
      $32\times 32$ & $1.0004$ & $-1$ & $1.0005$ & $-1$ & $1.0004$ & $-1$ & $1.0004$ & $-1$ \\ \hline
      $64\times 64$ & $1$      & $-1$ & $1$      & $-1$ & $1$ & $-1$ & $1$ & $-1$ \\ \hline
    \end{tabular}
  \end{center}
\end{table}

\begin{table}
  \begin{center}
    \caption{Maximum and minimum values of $u_0$ obtained with
    		different $\gamma$ for Example~\ref{exam5.2} on {\tt mesh90}.} \vspace{5pt}
    \label{tab:T2maxminu0-mesh90}
    \begin{tabular}{|c|c|c|c|c|c|c|c|c|c|c|c|c|}
      \hline
       & \multicolumn{2}{|c|}{$\gamma=20$}  & \multicolumn{2}{|c|}{$\gamma=40$} & \multicolumn{2}{|c|}{$\gamma=60$} & \multicolumn{2}{|c|}{$\gamma=99$}  \\ \hline
      Size & Max. & Min. & Max. & Min. & Max. & Min. & Max. & Min.  \\ \hline
      $8\times 8$   & $.995$ & $-.981$ & $.997$ & $-.981$ & $.999$ & $-.981$ & $1.007$& $-.981$  \\ \hline
      $16\times 16$ & $.997$ & $-.994$ & $.998$ & $-.994$ & $.999$ & $-.993$ & $.999$ & $-.993$ \\ \hline
      $32\times 32$ & $.999$ & $-.998$ & $.999$ & $-.998$ & $.999$ & $-.998$ & $.999$ & $-.998$ \\ \hline
      $64\times 64$ & $.999$ & $-.999$ & $.999$ & $-.999$ & $.999$ & $-.999$ & $.999$ & $-.999$ \\ \hline
    \end{tabular}
  \end{center}
\end{table}

\begin{figure}
  \begin{center}
    \includegraphics[width=5.4cm]{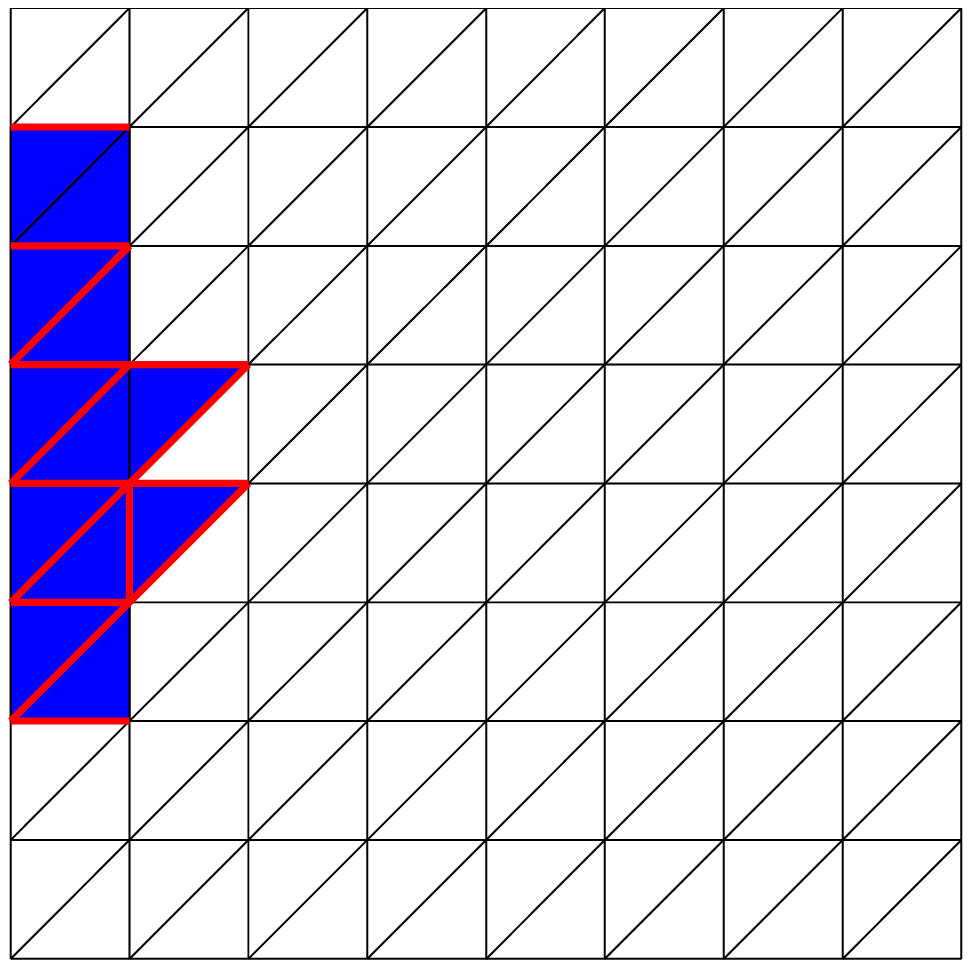} \hspace{-1cm}
    \includegraphics[width=5.4cm]{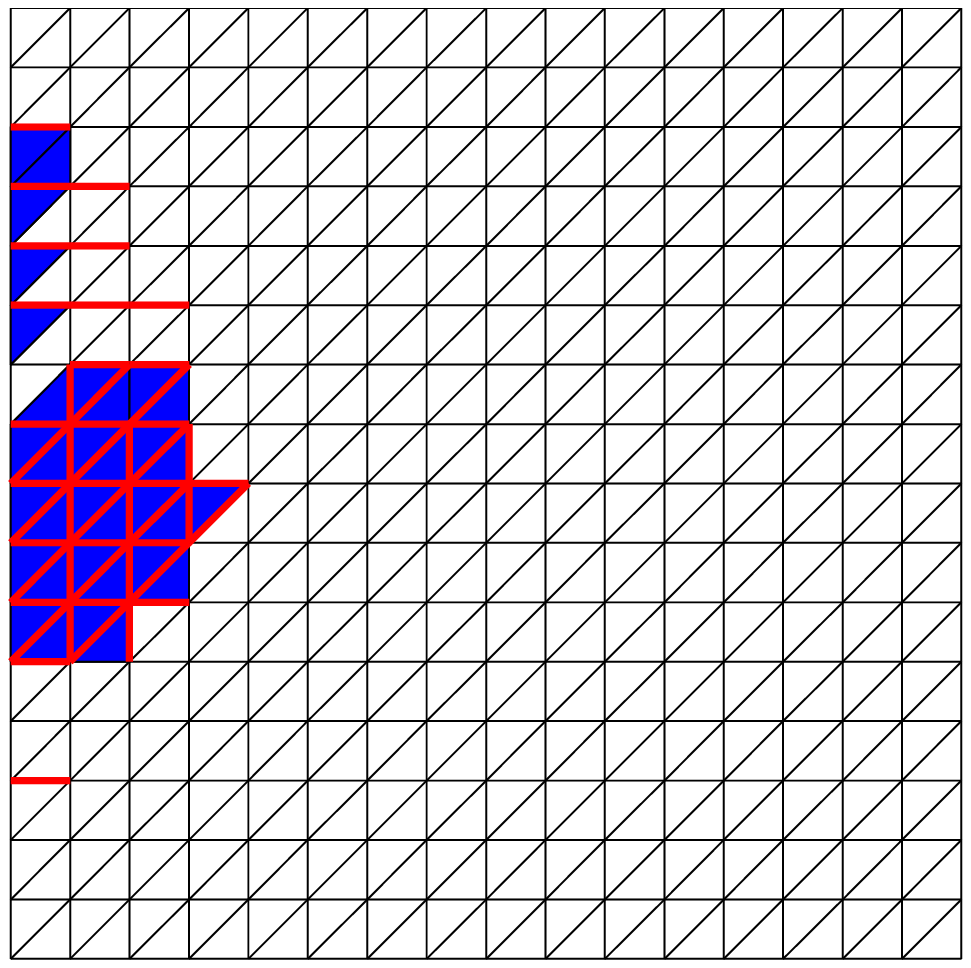} \hspace{-1cm}
    \includegraphics[width=5.4cm]{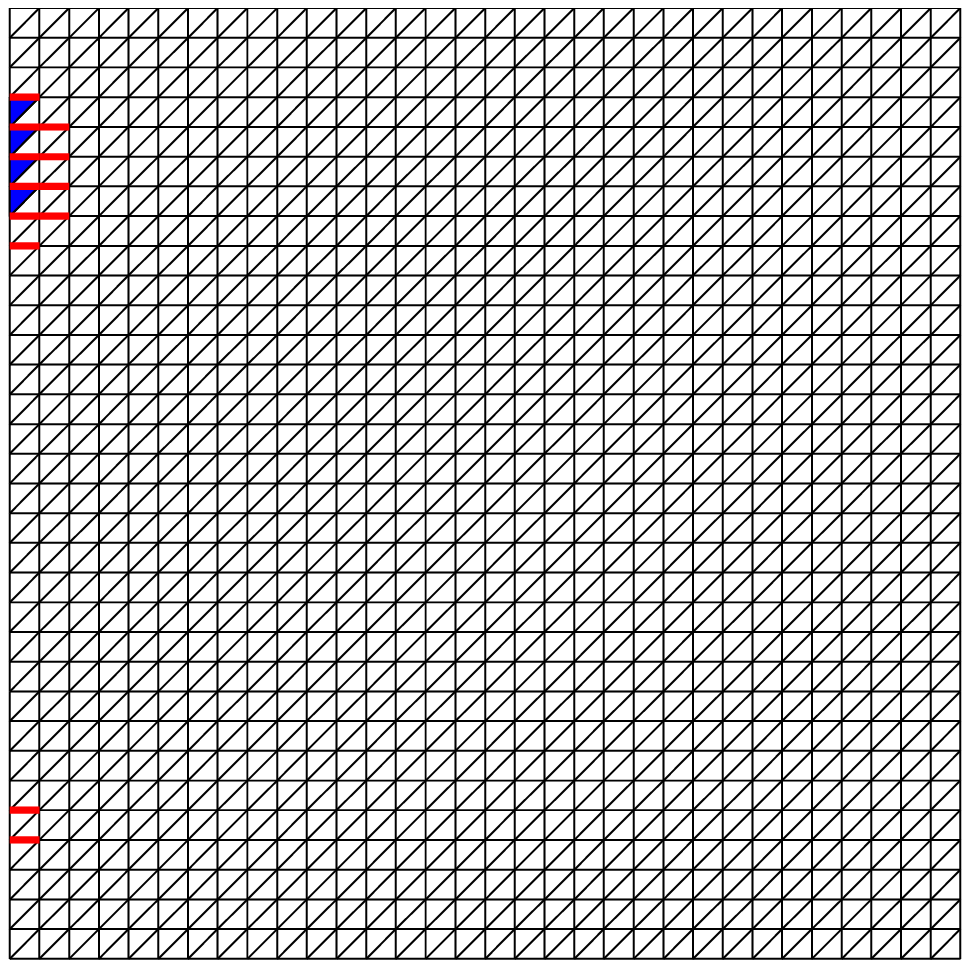}
  \caption{Triangles and edges on which overshoots are observed, using {\tt mesh45} with sizes $8\times 8$, $16\times 16$ and $32\times 32$ for Example~\ref{exam5.2}, with $\gamma=99$.}
  \label{fig:T2overshoot-mesh45}
  \end{center}
\end{figure}

\begin{figure}
  \begin{center}
    \includegraphics[width=5.4cm]{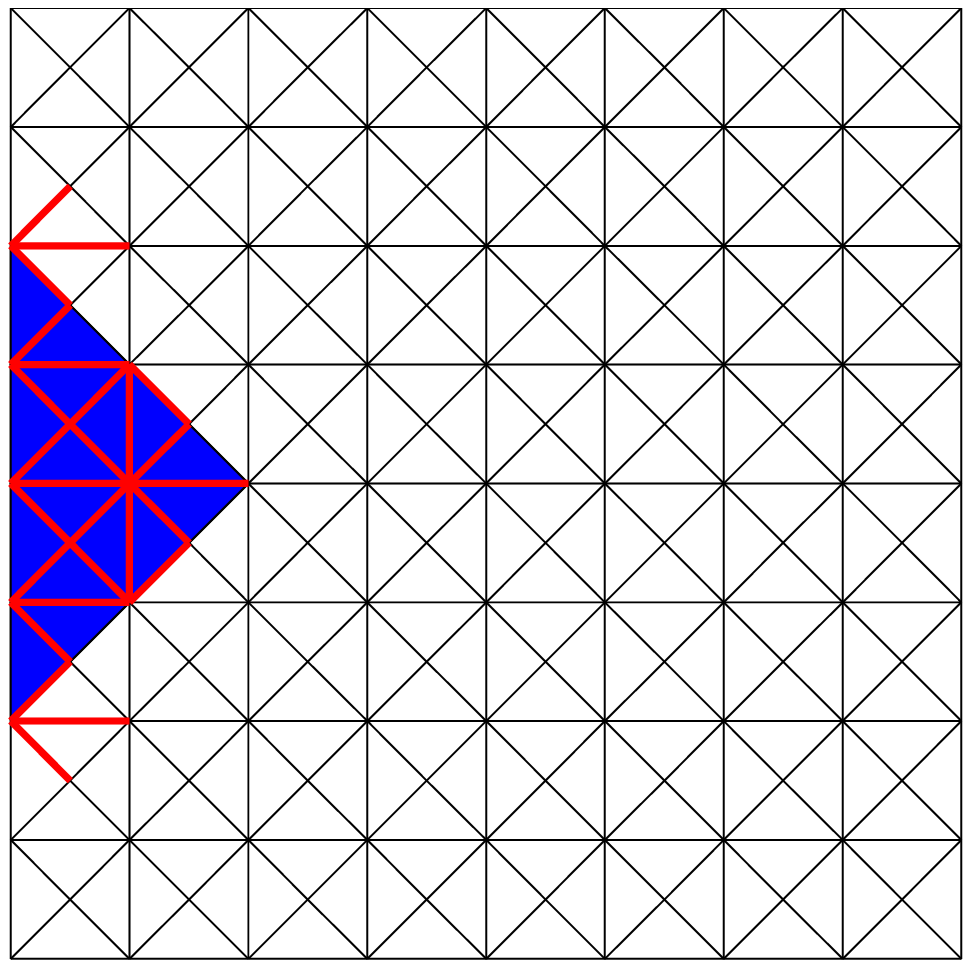} \hspace{-1cm}
    \includegraphics[width=5.4cm]{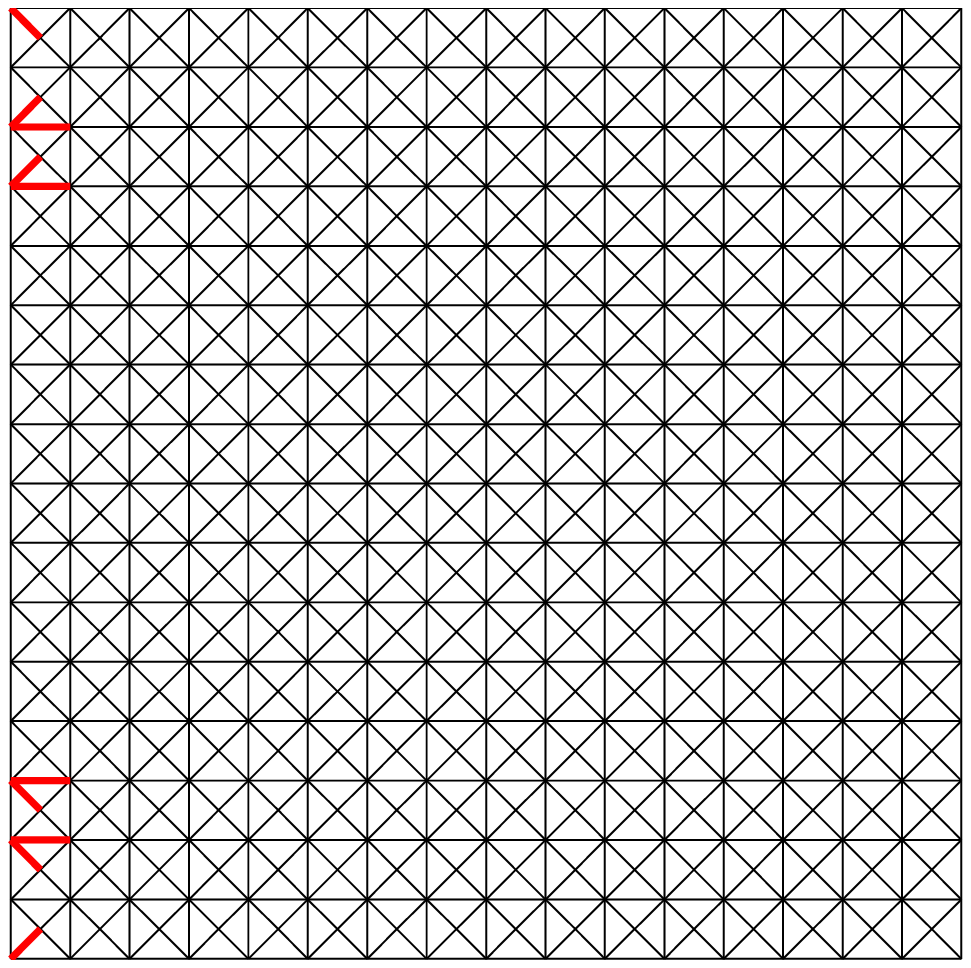} \hspace{-1cm}
    \includegraphics[width=5.4cm]{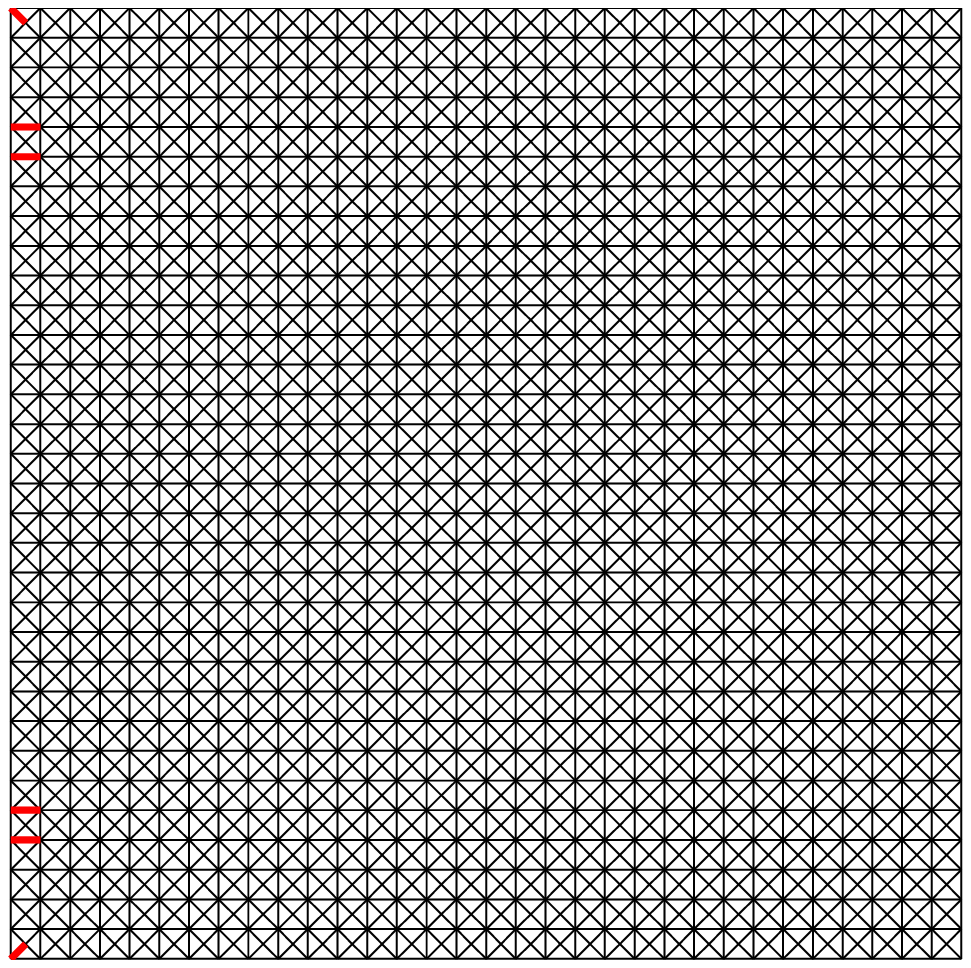}
  \caption{Triangles and edges on which overshoots are observed, using {\tt mesh90} with sizes $8\times 8$, $16\times 16$ and $32\times 32$ for Example~\ref{exam5.2}, with $\gamma=99$.}
  \label{fig:T2overshoot-mesh90}
  \end{center}
\end{figure}

\begin{figure}
  \begin{center}
    \includegraphics[width=5cm]{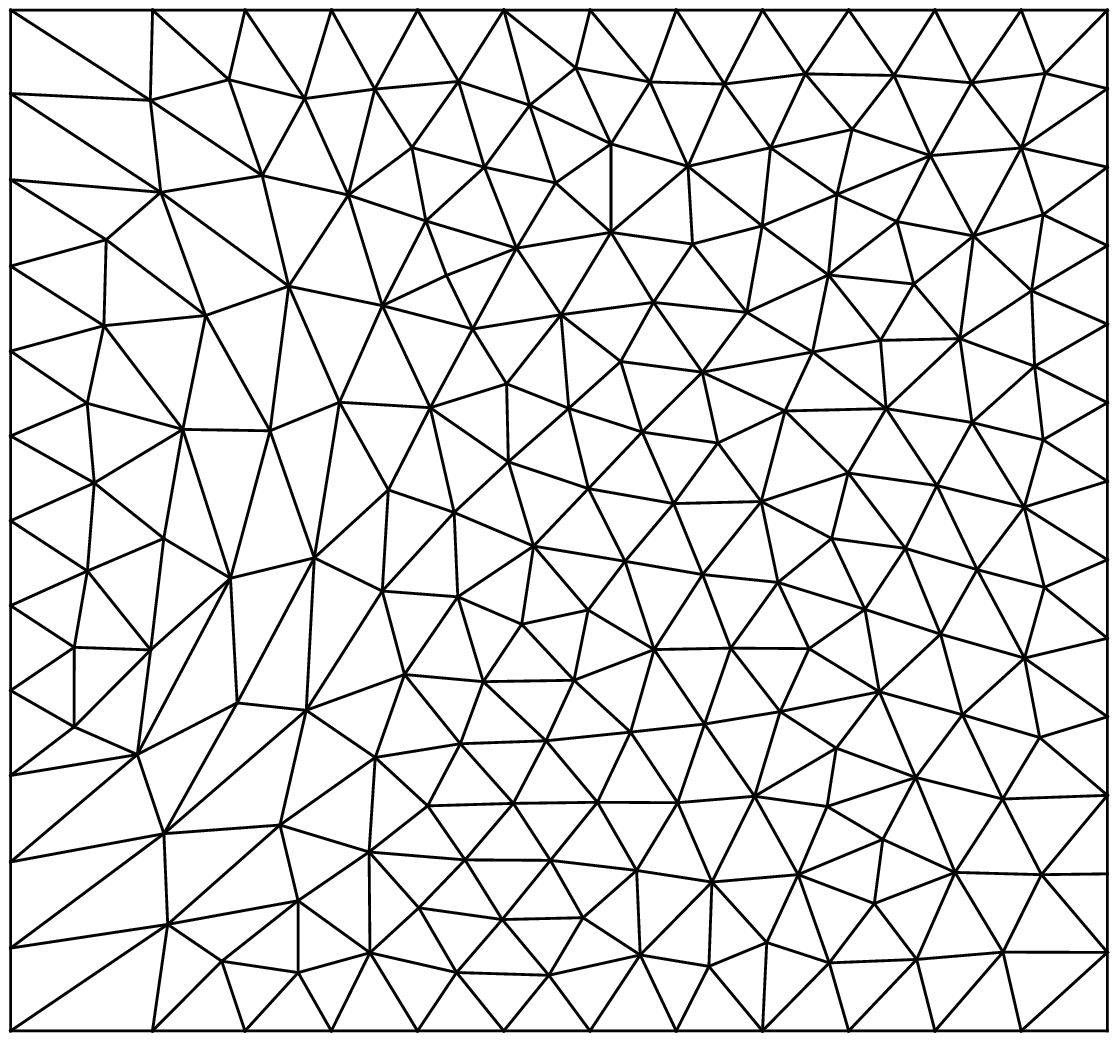} \hspace{-0.2cm}
    \includegraphics[width=5cm]{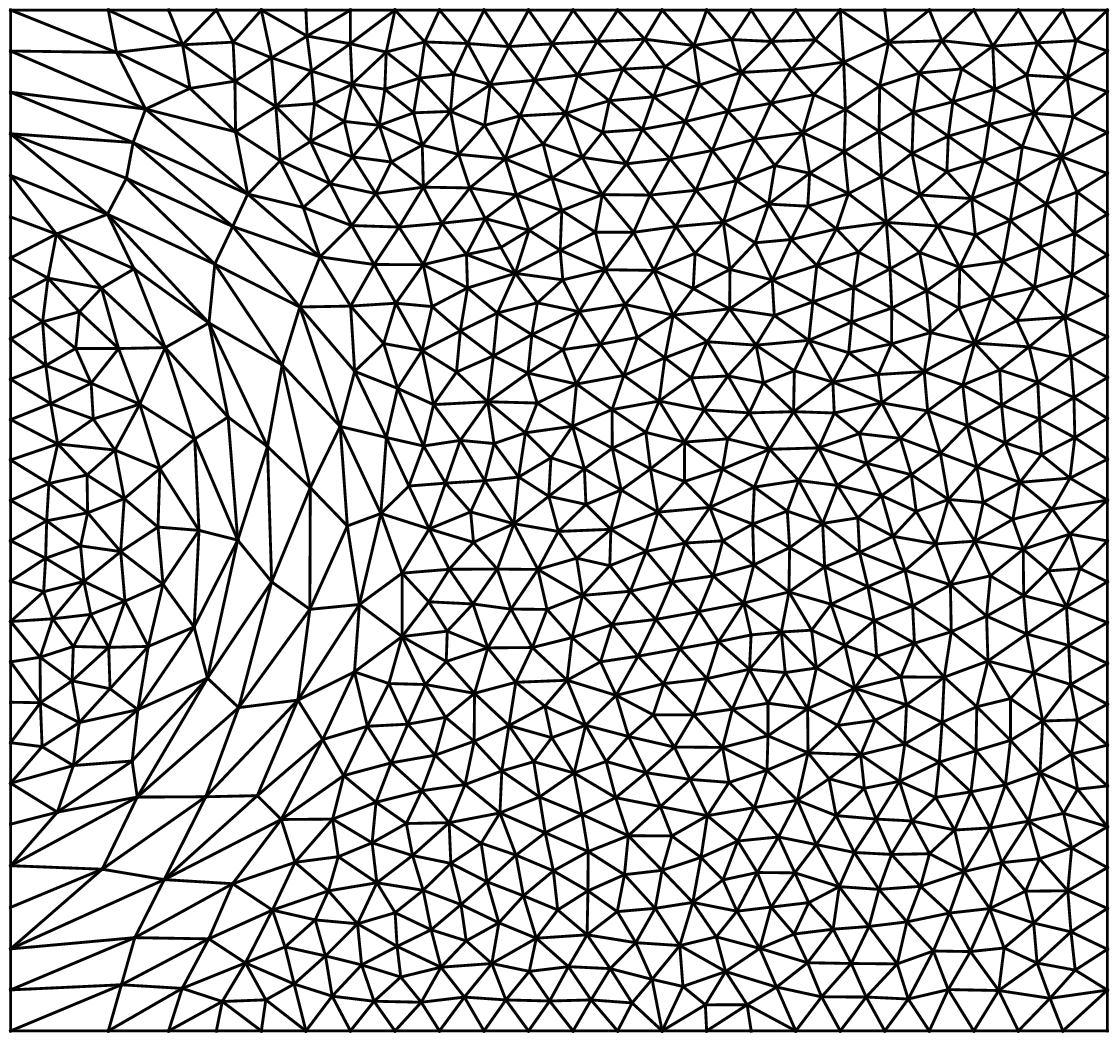}  \hspace{-0.2cm}
    \includegraphics[width=5cm]{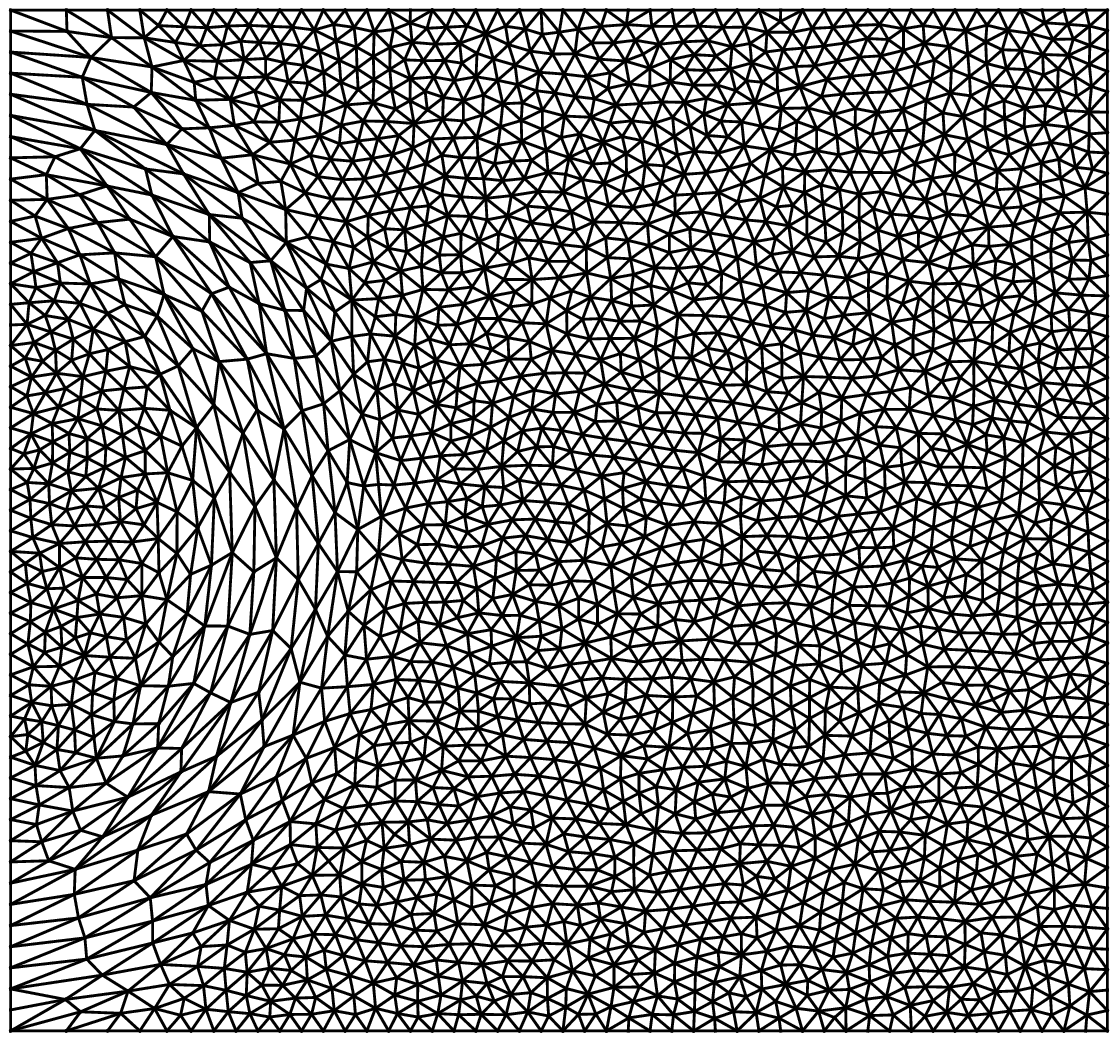}
  \caption{Meshes generated by BAMG with metric $\cA^{-1}$ and $\gamma=20$ for Example~\ref{exam5.2}.}
  \label{fig:BAMGmesh20}
  \end{center}
\end{figure}

%

\begin{figure}
  \begin{center}
    \includegraphics[width=5cm]{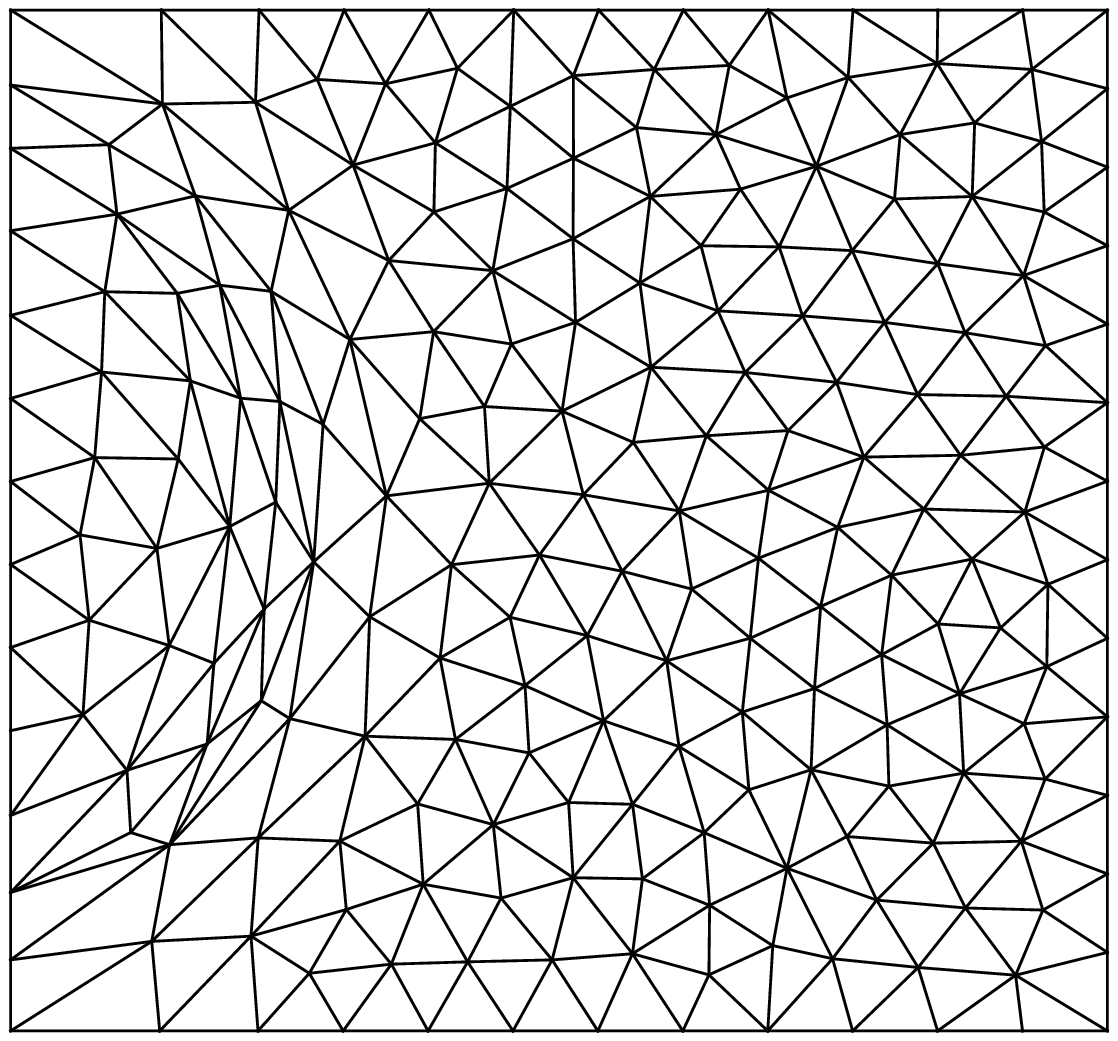} \hspace{-0.2cm}
    \includegraphics[width=5cm]{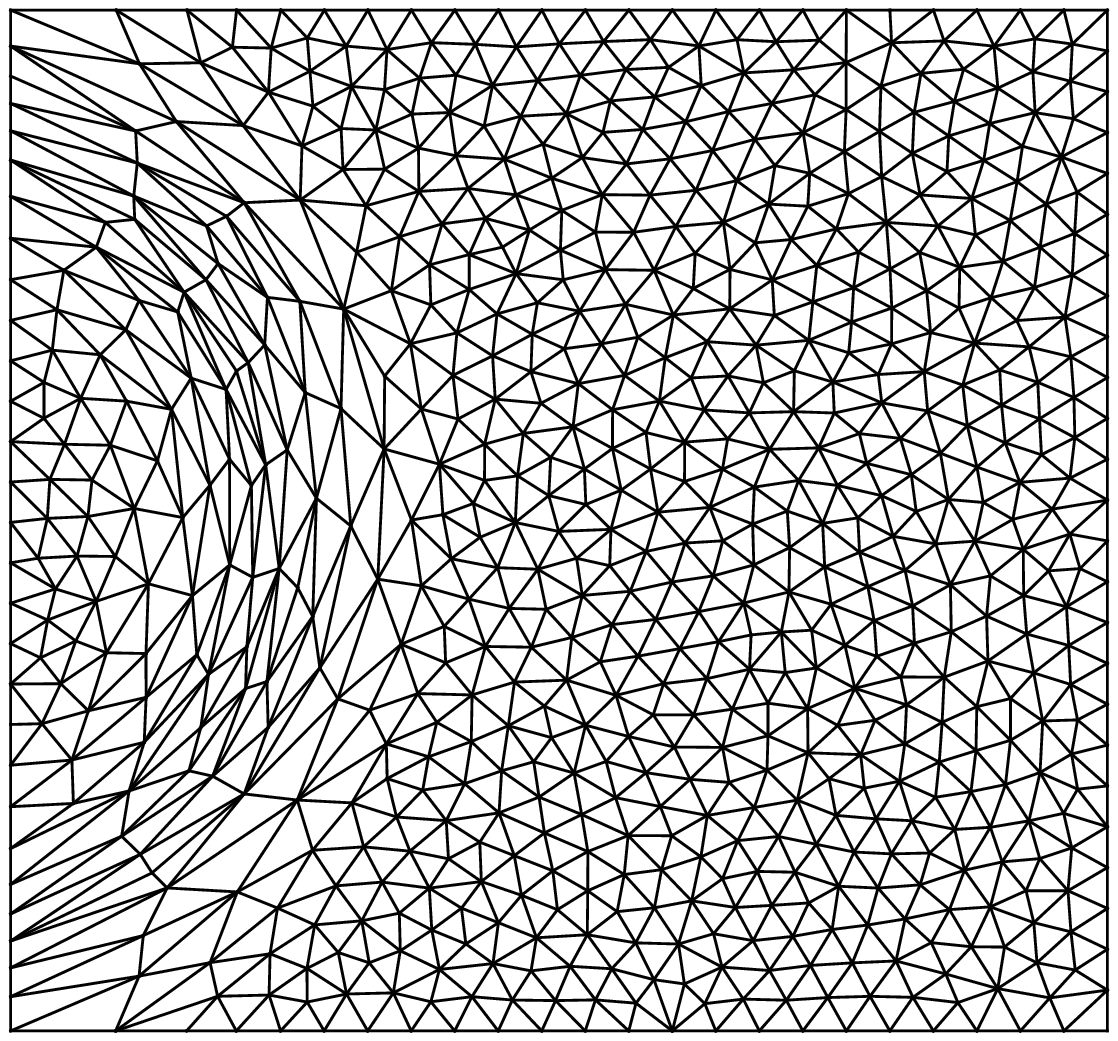} \hspace{-0.2cm}
    \includegraphics[width=5cm]{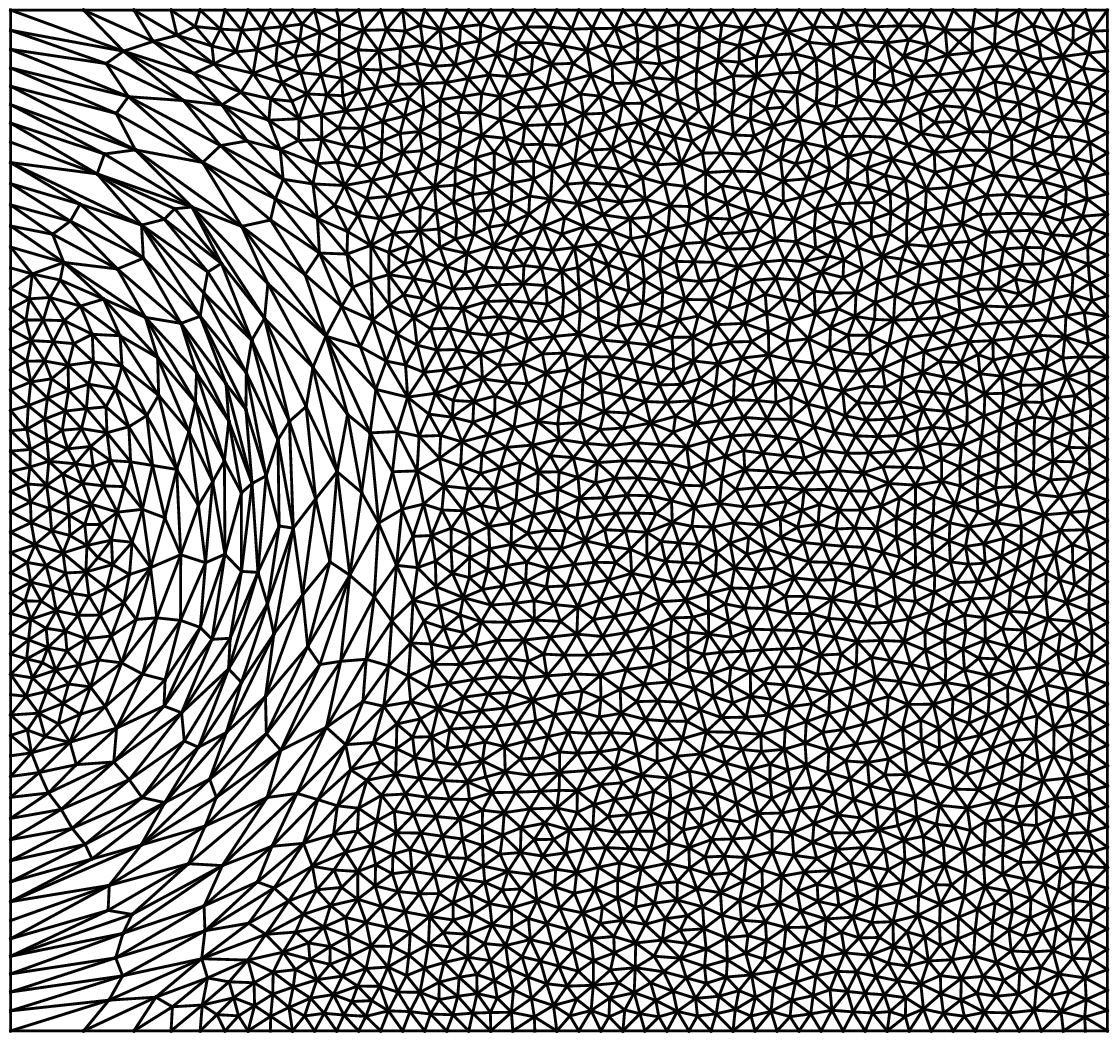}
  \caption{Meshes generated by BAMG with metric $\cA^{-1}$ and $\gamma=99$ for Example~\ref{exam5.2}.}
  \label{fig:BAMGmesh99}
  \end{center}
\end{figure}

\section{Conclusions}
\label{sec:conclusion}

In the previous sections we have studied the discrete maximum principle for a simplest and lowest order
weak Galerkin discretization of the anisotropic diffusion BVP (\ref{eq:pde}). The main results are
stated in Theorems~\ref{thm:pwconstant} and \ref{thm:general}.


Theorem~\ref{thm:pwconstant} states that the weak Galerkin approximation to
BVP (\ref{eq:pde}) satisfies a discrete maximum principle if the diffusion matrix $\cA$
is piecewise constant on the mesh and for any $K \in \T_h$, all of the angles of $K$
are nonobtuse when measured in the metric specified by $\cA^{-1}_K$, where $\cA_K$ is
the average of $\cA$ over $K$.
For the general anisotropic diffusion situation (cf. Theorem~\ref{thm:general}), the mesh is
required to be sufficiently fine, not very skewed (cf. (\ref{eq:thm:general-2})),
and $\mathcal{O}(h^2)$-acute (cf. (\ref{eq:thm:general-1})) when measured
in the metric specified by $\cA_K^{-1}$.  
These conditions are comparable to the mesh conditions for P1 conforming finite elements
in three and higher dimensions but stronger in two dimensions where
a Delaunay-type condition is sufficient to guarantee a P1 conforming FE
approximation to satisfy a discrete maximum principle.

Finally, it is worth pointing out that although the analysis has been carried out
in this work in two dimensions, it applies to three and higher dimensions without major modifications.

\vspace{20pt}

\noindent
{\bf Acknowledgment.}
This work was supported in part by the NSF under Grant DMS-1115118.


\begin{thebibliography}{10}

\bibitem{Ada75}
R.~A. Adams.
\newblock {\em Sobolev Spaces}.
\newblock Academic Press, New York, 1975.

\bibitem{Arnold85}
D.~Arnold and F.~Brezzi.
\newblock Mixed and nonconforming finite element methods: implementation,
  postprocessing and error estimates.
\newblock {\em RAIRO Mod\'{e}l. Math. Anal. Num\'{e}r.}, 19:7--32, 1985.

\bibitem{BP94}
A.~Berman and R.~J. Plemmons.
\newblock {\em Nonnegative Matrices in the Mathematical Sciences}.
\newblock Society for Industrial and Applied Mathematics, Philadelphia, 1994.

\bibitem{BKK08}
J.~Brandts, S.~Korotov, and M.~K\v{r}\'i\v{z}ek.
\newblock The discrete maximum principle for linear simplicial finite element
  approximations of a reaction-diffusion problem.
\newblock {\em Lin. Alg. Appl.}, 429:2344--2357, 2008.

\bibitem{BE04}
E.~Burman and A.~Ern.
\newblock Discrete maximum principle for {G}alerkin approximations of the
  {L}aplace operator on arbitrary meshes.
\newblock {\em C. R. Acad. Sci. Paris}, Ser.I 338:641--646, 2004.

\bibitem{Cia70}
P.~G. Ciarlet.
\newblock Discrete maximum principle for finite difference operators.
\newblock {\em Aequationes Math.}, 4:338--352, 1970.

\bibitem{Cia78}
P.~G. Ciarlet.
\newblock {\em The Finite Element Method for Elliptic Problems}.
\newblock North-Holland, Amsterdam, 1978.

\bibitem{CR73}
P.~G. Ciarlet and P.-A. Raviart.
\newblock Maximum principle and uniform convergence for the finite element
  method.
\newblock {\em Comput. Meth. Appl. Mech. Engrg.}, 2:17--31, 1973.

\bibitem{DDS04}
A.~Dr\v{a}g\v{a}nescu, T.~F. Dupont, and L.~R. Scott.
\newblock Failure of the discrete maximum principle for an elliptic finite
  element problem.
\newblock {\em Math. Comp.}, 74:1--23, 2004.

\bibitem{Eva98}
L.~C. Evans.
\newblock {\em Partial Differential Equations}.
\newblock American Mathematical Society, Providence, Rhode Island, 1998.
\newblock Graduate Studies in Mathematics, Volume 19.

\bibitem{Gu}
J.~Gu.
\newblock {\em Domain Decomposition Methods for Nonconforming Finite Element Discretizations}.
\newblock Nova Science Publishers, Inc., New York, 1999. 

\bibitem{Hec97}
F.~Hecht.
\newblock {BAMG} -- {B}idimensional {A}nisotropic {M}esh {G}enerator homepage.
\newblock {http://www.ann.jussieu.fr/$\sim$hecht/ftp/bamg/}, 1997.

\bibitem{Hoteit}
H. Hoteit, R. Mos\'{e}, B. Philippe, Ph. Ackerer and J. Erhel.
\newblock The maximum principle violations of the mixed-hybrid finite-element method applied to diffusion equations.
\newblock {\em Int. J. Numer. Meth. Engng.}, 55:1373--1390, 2002.

\bibitem{Hua11}
W.~Huang.
\newblock Discrete maximum principle and a {D}elaunay-type mesh condition for
  linear finite element approximations of two-dimensional anisotropic diffusion
  problems.
\newblock {\em Numer. Math. Theory Meth. Appl.}, 4:319--334, 2011.
\newblock (arXiv:1008.0562).

\bibitem{KK05}
J.~Kar{\'a}tson and S.~Korotov.
\newblock Discrete maximum principles for finite element solutions of nonlinear
  elliptic problems with mixed boundary conditions.
\newblock {\em Numer. Math.}, 99:669--698, 2005.

\bibitem{KKK07}
J.~Kar\'atson, S.~Korotov, and M.~K\v{r}\'i\v{z}ek.
\newblock On discrete maximum principles for nonlinear elliptic problems.
\newblock {\em Math. Comput. Sim.}, 76:99--108, 2007.

\bibitem{KSS09}
D.~Kuzmin, M.~J. Shashkov, and D.~Svyatskiy.
\newblock A constrained finite element method satisfying the discrete maximum
  principle for anisotropic diffusion problems.
\newblock {\em J. Comput. Phys.}, 228:3448--3463, 2009.

\bibitem{KL95}
M.~K\v{r}\'i\v{z}ek and Q.~Lin.
\newblock On diagonal dominance of stiffness matrices in {3D}.
\newblock {\em East-West J. Numer. Math.}, 3:59--69, 1995.

\bibitem{Let92}
F.~W. Letniowski.
\newblock Three-dimensional {D}elaunay triangulations for finite element
  approximations to a second-order diffusion operator.
\newblock {\em SIAM J. Sci. Stat. Comput.}, 13:765--770, 1992.

\bibitem{LH10}
X.~P. Li and W.~Huang.
\newblock An anisotropic mesh adaptation method for the finite element solution
  of heterogeneous anisotropic diffusion problems.
\newblock {\em J. Comput. Phys.}, 229:8072--8094, 2010.
\newblock (arXiv:1003.4530).

\bibitem{LiHu2012}
X.~P. Li and W.~Huang.
\newblock Maximum principle for the finite element solution of time dependent
  anisotropic diffusion problems.
\newblock {\em Numer Meth. P. D. E.}, 29:1963 -- 1985, 2013.
\newblock (arXiv:1209.5657).

\bibitem{LSS07}
X.~P. Li, D.~Svyatskiy, and M.~Shashkov.
\newblock Mesh adaptation and discrete maximum principle for {2D} anisotropic
  diffusion problems.
\newblock Technical Report LA-UR 10-01227, Los Alamos National Laboratory, Los
  Alamos, NM, 2007.

\bibitem{LSSV07}
K.~Lipnikov, M.~Shashkov, D.~Svyatskiy, and Y.~Vassilevski.
\newblock Monotone finite volume schemes for diffusion equations on
  unstructured triangular and shape-regular polygonal meshes.
\newblock {\em J. Comput. Phys.}, 227:492--512, 2007.

\bibitem{LS08}
R.~Liska and M.~Shashkov.
\newblock Enforcing the discrete maximum principle for linear finite element
  solutions of second-order elliptic problems.
\newblock {\em Comm. Comput. Phys.}, 3:852--877, 2008.

\bibitem{LuHuQi2012}
C.~Lu, W.~Huang, and J.~Qiu.
\newblock Maximum principle in linear finite element approximations of
  anisotropic diffusion-convection-reaction problems.
\newblock {\em Numer. Math.}, 127:515--537, 2014. 
\newblock (arXiv:1201.3564).

\bibitem{MD06}
M.~J. Mlacnik and L.~J. Durlofsky.
\newblock Unstructured grid optimization for improved monotonicity of discrete
  solutions of elliptic equations with highly anisotropic coefficients.
\newblock {\em J. Comput. Phys.}, 216:337--361, 2006.

\bibitem{WG-biharmonic}
L.~Mu, J.~Wang, Y.~Wang, and X.~Ye.
\newblock A weak Galerkin mixed finite element method for biharmonic equations.
\newblock In O.P.~Iliev et.al., editors, 
  {\em Numerical Solution of Partial Differential Equations: Theory, Algorithms, and Their Applications}, 
  volume 45 of {\em Springer Proceedings in Mathematics \& Statistics}, New York, 2013. Springer-Verlag.
\newblock (arXiv:1210.3818).

\bibitem{MuWangWangYe}
L.~Mu, J.~Wang, Y.~Wang, and X.~Ye.
\newblock A computational study of the weak Galerkin method for second order
  elliptic equations.
\newblock {\em Numer. Alg.}, 63:753--777, 2013.
\newblock (arXiv:1111.0618).

\bibitem{mwy-wg-stabilization}
L.~Mu, J.~Wang, and X.~Ye.
\newblock Weak Galerkin finite element methods on polytopal meshes.
\newblock {\em Int. J. Numer. Anal. Model.}, 12:31--53, 2015.
\newblock (arXiv:1204.3655).

\bibitem{RT77}
P.~Raviart and J.~Thomas.
\newblock A mixed finite element method for second order elliptic problems.
\newblock In I.~Galligani and E.~Magenes, editors, {\em Mathematical Aspects of
  the Finite Element Method}, volume 606 of {\em Lectures Notes in
  Mathematics}, New York, 1977. Springer-Verlag.

\bibitem{ShYu2011}
Z.~Sheng and G.~Yuan.
\newblock The finite volume scheme preserving extremum principle for diffusion
  equations on polygonal meshes.
\newblock {\em J. Comput. Phys.}, 230:2588--2604, 2011.

\bibitem{Sto82}
G.~Stoyan.
\newblock On a maximum principle for matrices, and on conservation of
  monotonicity. {W}ith applications to discretization methods.
\newblock {\em Z. Angew. Math. Mech.}, 62:375--381, 1982.

\bibitem{Sto86}
G.~Stoyan.
\newblock On maximum principles for monotone matrices.
\newblock {\em Lin. Alg. Appl.}, 78:147--161, 1986.

\bibitem{SF73}
G.~Strang and G.~J. Fix.
\newblock {\em An Analysis of the Finite Element Method}.
\newblock Prentice Hall, Englewood Cliffs, NJ, 1973.

\bibitem{Var66}
R.~S. Varga.
\newblock On a discrete maximum principle.
\newblock {\em SIAM J. Numer. Anal.}, 3:355--359, 1966.

\bibitem{Vohralik}
M. Vohral\'{i}k and B.I. Wohlmuth.
\newblock Mixed finite element methods: implementation with one unknown per element, local flux expressions, positivity, polygonal meshes, and relations to other methods.
\newblock {\em Mathematical Models and Methods in Applied Sciences}, 23:803--838, 2013.

\bibitem{wy-mixed}
J.~Wang and X.~Ye.
\newblock A weak Galerkin mixed finite element method for second-order elliptic problems.
\newblock {\em Math. Comp.}, 83:2101--2126, 2014.
\newblock (arXiv:1202.3655).

\bibitem{WangYe_PrepSINUM_2011}
J.~Wang and X.~Ye.
\newblock A weak Galerkin finite element method for second-order elliptic
  problems.
\newblock {\em J. Comput. Appl. Math.}, 241:103--115, 2013.
\newblock (arXiv:1104.2897).

\bibitem{WaZh11}
J.~Wang and R.~Zhang.
\newblock Maximum principle for {P1}-conforming finite element approximations
  of quasi-linear second order elliptic equations.
\newblock {\em SIAM J. Numer. Anal.}, 50:626--642, 2012.
\newblock (arXiv:1105.1466).

\bibitem{XZ99}
J.~Xu and L.~Zikatanov.
\newblock A monotone finite element scheme for convection-diffusion equations.
\newblock {\em Math. Comput.}, 69:1429--1446, 1999.

\bibitem{YuSh2008}
G.~Yuan and Z.~Sheng.
\newblock Monotone finite volume schemes for diffusion equations on polygonal
  meshes.
\newblock {\em J. Comput. Phys.}, 227:6288--6312, 2008.

\bibitem{ZZS2013}
Y.~Zhang, X.~Zhang, and C.-W. Shu.
\newblock Maximum-principle-satisfying second order discontinuous {G}alerkin
  schemes for convection-diffusion equations on triangular meshes.
\newblock {\em J. Comput. Phys.}, 234:295 -- 316, 2013.

\end{thebibliography}

\end{document}